\pdfoutput=1
\documentclass[11pt,a4paper]{article}
\usepackage[margin=2.4cm]{geometry}
\usepackage{amsmath,amssymb,amsthm}
\usepackage{booktabs}
\usepackage{graphicx}
\usepackage{tikz}
\usetikzlibrary{arrows.meta,shapes.misc}
\usepackage[hidelinks]{hyperref}

\newtheorem{theorem}{Theorem}[section]
\newtheorem{lemma}[theorem]{Lemma}
\newtheorem{proposition}[theorem]{Proposition}
\newtheorem{corollary}[theorem]{Corollary}
\theoremstyle{definition}
\newtheorem{definition}[theorem]{Definition}

\newtheorem{example}[theorem]{Example}

\newcommand{\Aut}{\operatorname{Aut}}
\newcommand{\cP}{\mathcal{P}}
\newcommand{\rise}[2]{#1^{\overline{#2}}}

\title{Exact and asymptotic enumeration of unrestricted binary phylogenetic networks through automorphism weights}
\author{Josep Batle\thanks{E-mail: batlequantum@gmail.com}\\[4pt]
\small CRISP -- Centre de Recerca Independent de sa Pobla, 07420 sa Pobla, Balearic Islands, Spain}
\date{}

\begin{document}
\maketitle

\begin{abstract}
Let $\cP_{\ell,k}$ be the set of rooted binary phylogenetic networks with $\ell$ labelled leaves, $k$ reticulations and no parallel edges. We write $|\cP_{\ell,k}|=W_k(\ell)+D_k(\ell)$, where the weighted count $W_k(\ell)$ adds the inverse orders of the leaf-fixing automorphism groups and the defect $D_k(\ell)$ collects the remainder. The weighted count satisfies, for every $k$, a recursion over the source layers of the tree-component structure that involves neither a list of component graphs nor any distinction between symmetric and asymmetric configurations, and its exponential generating function is a Laurent polynomial in $\sqrt{1-2x}$. Automorphism groups of networks are $2$-groups, elementary abelian for $k\le5$ but not in general. For $k\le5$ the defect is the weighted count of networks with a distinguished involution, which obeys an extension of the same recursion. An exact symbolic evaluation of the two recursions yields $|\cP_{\ell,k}|$ in closed form for $k\le5$, the case $k=5$ being new; it reproduces the published counts for $k\le4$ and corrects a coefficient in a published generating function for $k=3$. For every $k$, uniformly over explicit ranges of $k$, we prove that the non-tree-child networks are a fraction $2k(k-1)/\ell$ of the tree-child networks to leading order, which gives the third term of the asymptotic expansion of $|\cP_{\ell,k}|$. We also prove that reticulation-visible networks exceed tree-child networks by the fraction $k(k-1)/\ell$, and that a uniformly random network in $\cP_{\ell,k}$ has a nontrivial automorphism with probability $k(k-1)/(4\ell^3)$ to leading order.
\end{abstract}



\section{Introduction}\label{sec:intro}

A rooted binary phylogenetic network is a directed acyclic graph that records an evolutionary history in which, besides speciation, two lineages may merge in a reticulation event such as hybridisation or lateral gene transfer. Counting such networks is a natural problem of enumerative combinatorics and a prerequisite for understanding the size of the search spaces explored by reconstruction methods. When the number $k$ of reticulations is not restricted, the number of networks with $\ell$ labelled leaves is infinite. It is finite for fixed $k$, and the problem is to determine the number $|\cP_{\ell,k}|$ of networks with $\ell$ leaves and $k$ reticulations as a function of $\ell$.

For tree-child networks, in which every non-leaf vertex has a child that is not a reticulation, the problem is solved. Pons and the present author conjectured an exact formula in terms of a class of constrained words \cite{PonsBatle21}, which was proved by Lin, Liu, Liu, Liu and Xin \cite{Lin26}. For unrestricted networks the exact counts are known only for few reticulations. Zhang obtained $|\cP_{\ell,1}|$ \cite{Zhang19}. Mansouri obtained the exponential generating functions of vertex-labelled networks for $k\le3$ and the leaf-labelled count for $k=2$ \cite{Mansouri22}. Yu, Fuchs, Yu and Zhang obtained the leaf-labelled count for $k=3$ \cite{Yu26}, and Yu and Zhang the count for $k=4$ \cite{YuZhang26}. The last two results use the component-graph method of Gunawan, Rathin and Zhang \cite{GRZ20}. Every network is mapped to the multigraph obtained by contracting its tree components, the possible component graphs with $k+1$ vertices are listed, and the networks attached to each are counted. Within each group the count distinguishes configurations that are invariant under the exchange of two designated vertices from those that are not. Both the list and the distinction grow rapidly with $k$.

The asymptotic behaviour for fixed $k$ is known to first order. Fuchs, Gittenberger and Mansouri \cite{FGM19} and Fuchs, Huang and Yu \cite{FHY} showed that tree-child networks, normal networks and unrestricted networks have the same first-order asymptotics, with the constant $2^{k-1}\sqrt2/k!$. Yu and Zhang extended this to $k=o(\sqrt\ell)$ and computed further terms of the tree-child expansion \cite{YZasym}. Related first-order results concern normal and hybridization networks \cite{FSZ}, galled tree-child networks \cite{CFY}, and galled and reticulation-visible networks \cite{ChangFuchs24}.

The present contribution separates the two sources of difficulty in the exact problem. Every network $N$ carries the group $\Aut(N)$ of its automorphisms fixing all leaves. The count $|\cP_{\ell,k}|$ splits into a weighted count $W_k(\ell)=\sum_N|\Aut(N)|^{-1}$ and a defect $D_k(\ell)=\sum_N(1-|\Aut(N)|^{-1})$. The weighted count coincides, up to an explicit factor, with the vertex-labelled count of \cite{Mansouri22}. It satisfies a recursion over the source layers of the tree-component structure in which no symmetric case analysis occurs and no component graph is listed. The defect is supported on symmetric networks, which are rare. For $k\le5$ every automorphism is an involution or the identity, and the defect becomes the weighted count of networks with a distinguished involution, to which the recursion extends. The result is a closed form for $|\cP_{\ell,k}|$ for every $k\le5$. The closed forms with four and five reticulations are obtained by an exact symbolic evaluation of the recursions, so this part of the work is computer-assisted; the evaluation is checked against exhaustive enumeration, against the published counts and against the asymptotic results. The new case $k=5$ begins with $1896$, $86453$, $3162549$ and $109793685$ networks for $\ell=1,2,3,4$.

On the asymptotic side, the decomposition of non-tree-child networks according to their violations of the tree-child condition gives three results for every $k$, uniformly over explicit ranges of $k$. The first is the excess of unrestricted over tree-child networks, which yields the third term of the expansion of $|\cP_{\ell,k}|$. The second is the excess of reticulation-visible over tree-child networks. The third is the probability that a random network is symmetric.

The present contribution is divided as follows. Section~\ref{sec:weighted} introduces the weighted count and the defect and relates the former to vertex-labelled counts. Section~\ref{sec:aut} analyses leaf-fixing automorphisms, shows that their groups are $2$-groups, and evaluates $D_2$ and $D_3$ in closed form. Section~\ref{sec:recursion} derives the layer-peeling recursion for the weighted count and its generating-function form. Section~\ref{sec:verification} compares the recursion with the published counts for at most three reticulations and identifies a coefficient error in a published generating function. Section~\ref{sec:asympt} proves the asymptotic results for every number of reticulations. Section~\ref{sec:symrec} extends the recursion to networks with a distinguished involution. Section~\ref{sec:k45} obtains the exact counts with four and five reticulations. Section~\ref{sec:beyond} describes what an extension to six or more reticulations requires. Appendix~\ref{app:k4} classifies the symmetric networks with four reticulations by hand, independently of the recursion. Finally, some conclusions are drawn in Sec.~\ref{sec:conclusions}.

\section{The weighted count and the defect}\label{sec:weighted}

\begin{definition}\label{def:network}
A binary phylogenetic network on $[\ell]=\{1,\dots,\ell\}$ is a finite directed acyclic graph without parallel edges with the following properties. One vertex, the root, has in-degree $0$ and out-degree $1$. Exactly $\ell$ vertices, the leaves, have in-degree $1$ and out-degree $0$ and are labelled bijectively by $[\ell]$. Every other vertex is either a tree node, with in-degree $1$ and out-degree $2$, or a reticulation, with in-degree $2$ and out-degree $1$. The set of such networks with $k$ reticulations is denoted by $\cP_{\ell,k}$.
\end{definition}

Equating the sums of in-degrees and out-degrees shows that a network in $\cP_{\ell,k}$ has $\ell+k-1$ tree nodes and $2\ell+3k-1$ edges \cite[Prop.~2.1]{YuZhang26}. Removing the two in-edges of every reticulation decomposes a network into vertex-disjoint trees, its tree components, rooted at the root and at the reticulations. We write $t_\ell=(2\ell-3)!!$ for the number of rooted binary phylogenetic trees on $[\ell]$, with $t_1=1$. We write $F_{n,a}$ for the number of forests of $a$ rooted binary trees whose leaf sets partition $[n]$:
\begin{equation}\label{eq:forests}
F_{n,a}=\frac{(2n-a-1)!}{2^{\,n-a}\,(n-a)!\,(a-1)!}\qquad (n\ge a\ge 1),
\end{equation}
with $F_{0,0}=1$ and $F_{n,a}=0$ otherwise \cite[Prop.~2.8.1]{SempleSteel03}. Every tree in such a forest has a root edge, so a forest counted by $F_{n,a}$ has $2n-a$ edges.

\begin{definition}\label{def:aut}
An automorphism of a network $N$ is a bijection of the vertex set of $N$ onto itself that maps edges onto edges and fixes every leaf. The automorphisms form a group $\Aut(N)$. Every automorphism fixes the root, the only vertex of in-degree $0$.
\end{definition}

\begin{definition}\label{def:WD}
For $\ell\ge 1$ and $k\ge 0$ the weighted count and the defect are
\begin{equation}\label{eq:WD}
W_k(\ell)=\sum_{N\in\cP_{\ell,k}}\frac{1}{|\Aut(N)|},\qquad
D_k(\ell)=\sum_{N\in\cP_{\ell,k}}\Bigl(1-\frac{1}{|\Aut(N)|}\Bigr),
\end{equation}
so that
\begin{equation}\label{eq:split}
|\cP_{\ell,k}|=W_k(\ell)+D_k(\ell).
\end{equation}
\end{definition}

The defect vanishes on every network with trivial automorphism group. In particular it vanishes on all tree-child networks, because every vertex of a tree-child network lies on a path of in-degree-one vertices ending at a fixed leaf, and is therefore fixed.

Deleting the root and its out-edge from $N\in\cP_{\ell,k}$ yields a directed graph $N^\circ$ whose top vertex has out-degree $2$. This is the convention of \cite{Mansouri22}. The graph $N^\circ$ has $n=2\ell+2k-1$ vertices. A vertex labelling of $N^\circ$ is a bijection from its vertex set onto $[n]$. Two vertex-labelled graphs are identified when a label-preserving isomorphism exists. Let $G_{k,n}$ denote the number of vertex-labelled graphs obtained in this way from networks with $k$ reticulations.

\begin{lemma}\label{lem:vertexlabelled}
For $\ell\ge1$, $k\ge0$ and $n=2\ell+2k-1$,
\begin{equation}\label{eq:GW}
G_{k,n}=\binom{n}{\ell}\,(n-\ell)!\;W_k(\ell).
\end{equation}
\end{lemma}

\begin{proof}
A vertex-labelled graph determines a leaf-labelled network: the leaves are relabelled by the ranks of their labels. Conversely, fix $N\in\cP_{\ell,k}$. The vertex-labelled graphs that produce $N$ are obtained by choosing the set of $\ell$ leaf labels, in $\binom{n}{\ell}$ ways, and assigning the remaining $n-\ell$ labels bijectively to the internal vertices. Two such assignments give the same vertex-labelled graph if and only if they differ by an element of $\Aut(N)$. The group acts freely on the assignments, since an automorphism that preserves an assignment fixes every vertex. The number of distinct assignments is therefore $(n-\ell)!/|\Aut(N)|$. Summing over $N$ gives \eqref{eq:GW}.
\end{proof}

Let $G'_k(z)=\sum_n G_{k,n}z^n/n!$ be the exponential generating function of vertex-labelled networks without parallel edges. Since $n$ is odd, $G'_k(z)=z\,\widehat G_k(z^2)$ for a power series $\widehat G_k$. Lemma~\ref{lem:vertexlabelled} gives
\begin{equation}\label{eq:WfromG}
W_k(\ell)=\ell!\,[x^{\ell+k-1}]\,\widehat G_k(x).
\end{equation}
Mansouri \cite{Mansouri22} obtains $\widehat G_k$ for $k\le3$ in the form
\begin{equation}\label{eq:Mform}
\widehat G_k(x)=\frac{a_k(x)-b_k(x)\sqrt{1-2x}}{(1-2x)^{2k-1/2}}.
\end{equation}
The published polynomials for networks without parallel edges are
\begin{align}
a_2(x)&=x^4+\tfrac12x^2+\tfrac32x, & b_2(x)&=x^2+\tfrac32x, \label{eq:M2}\\
a_3(x)&=3x^6+2x^5+4x^4+2x^3+\tfrac{69}{4}x^2, & b_3(x)&=x^5+\tfrac92x^4+11x^3+\tfrac{69}{4}x^2. \label{eq:M3}
\end{align}

\section{Leaf-fixing automorphisms}\label{sec:aut}

Throughout this section $\sigma$ is an automorphism of a network $N$. A vertex $x$ is moved if $\sigma(x)\neq x$ and fixed otherwise.

\begin{lemma}\label{lem:seams}
Let $x$ be a moved vertex and $y$ a fixed vertex.
(i) If $(x,y)$ is an edge, then $y$ is a reticulation whose parents are $x$ and $\sigma(x)$.
(ii) If $(y,x)$ is an edge, then $y$ is a tree node whose children are $x$ and $\sigma(x)$.
\end{lemma}

\begin{proof}
For (i), $\sigma$ maps the edge $(x,y)$ onto $(\sigma(x),y)$. Since $\sigma(x)\neq x$, the vertex $y$ has two distinct parents, so it is a reticulation and its parent set is $\{x,\sigma(x)\}$. For (ii), $\sigma$ maps $(y,x)$ onto $(y,\sigma(x))$, so $y$ has two distinct children $x$ and $\sigma(x)$ and is a tree node.
\end{proof}

A fixed reticulation with a moved parent is called a seam of $\sigma$.

\begin{lemma}\label{lem:seed}
Let $\sigma\neq\mathrm{id}$. There is a moved vertex $v$ all of whose children are fixed. With $w=\sigma(v)$ one has $\sigma(w)=v$, and exactly one of the following holds:
(T) $v$ and $w$ are tree nodes with the same two children $r_1,r_2$, both reticulations with parent set $\{v,w\}$;
(R) $v$ and $w$ are reticulations with the same child $c$, a reticulation with parent set $\{v,w\}$.
\end{lemma}

\begin{proof}
Order the vertices topologically and let $v$ be the last moved vertex. Its children come later in the order and are therefore fixed. The vertex $v$ is not a leaf, because leaves are fixed. By Lemma~\ref{lem:seams}(i) every child of $v$ is a reticulation with parent set $\{v,w\}$. Applying $\sigma$, the vertex $\sigma(w)$ is also a parent of every child of $v$, so $\sigma(w)\in\{v,w\}$. Since $\sigma$ is injective and $\sigma(v)=w$, it follows that $\sigma(w)=v$. If $v$ is a tree node, its two children form the pair of case (T), and $w$, which has the same degree, has the same children. If $v$ is a reticulation, its child is the vertex $c$ of case (R).
\end{proof}

The proof of Lemma~\ref{lem:seed} rests on a single observation, illustrated in Fig.~\ref{fig:seeds}, which is also the starting point of the two lemmas that follow. If $x$ is the last moved vertex in a topological order, its children are fixed, and by Lemma~\ref{lem:seams}(i) each child is a reticulation whose parents are exactly the elements of the orbit of $x$. Since a reticulation has two parents, that orbit has at most two elements. The chain of implications is
\begin{center}
\begin{tabular}{c}
$\sigma\neq\mathrm{id}$\\
$\Downarrow$\\
$x$ is the last moved vertex in a topological order\\
$\Downarrow$\\
the children of $x$ are fixed\\
$\Downarrow$\\
each child of $x$ is a reticulation whose parents form the orbit of $x$\\
$\Downarrow$\\
the orbit of $x$ has at most two elements.
\end{tabular}
\end{center}
For an automorphism of prime order every moved vertex has an orbit whose size is that prime, which must therefore be $2$; this is Lemma~\ref{lem:autbound}. For an automorphism of order four, the same reasoning applied to the last vertex with an orbit of four elements shows that its children are reticulations with orbits of two elements, and a count of the reticulations below them gives at least six; this is Lemma~\ref{lem:exponent2}.

\begin{figure}[ht]
\centering
\begin{tikzpicture}[>={Stealth[length=2mm]},
  tn/.style={circle,draw,fill=white,inner sep=0pt,minimum size=6pt},
  rt/.style={circle,draw,fill=gray!45,inner sep=0pt,minimum size=6pt},
  lab/.style={font=\footnotesize}]
 \node[tn,label={[lab]left:$v$}] (v) at (0,2.0) {};
 \node[tn,label={[lab]right:$w$}] (w) at (1.6,2.0) {};
 \node[rt,label={[lab]left:$r_1$}] (r1) at (0,0.9) {};
 \node[rt,label={[lab]right:$r_2$}] (r2) at (1.6,0.9) {};
 \draw[->] (0,2.6) -- (v); \draw[->] (1.6,2.6) -- (w);
 \draw[->] (v) -- (r1); \draw[->] (v) -- (r2); \draw[->] (w) -- (r1); \draw[->] (w) -- (r2);
 \draw[->] (r1) -- (0,0.3); \draw[->] (r2) -- (1.6,0.3);
 \draw[<->,dashed] (0.2,2.25) to[bend left=30] node[lab,above] {$\sigma$} (1.4,2.25);
 \node[lab] at (0.8,-0.3) {(T)};
 \node[rt,label={[lab]left:$v$}] (v2) at (4.4,2.0) {};
 \node[rt,label={[lab]right:$w$}] (w2) at (6.0,2.0) {};
 \node[rt,label={[lab]right:$c$}] (c) at (5.2,0.9) {};
 \draw[->] (4.1,2.6) -- (v2); \draw[->] (4.7,2.6) -- (v2);
 \draw[->] (5.7,2.6) -- (w2); \draw[->] (6.3,2.6) -- (w2);
 \draw[->] (v2) -- (c); \draw[->] (w2) -- (c); \draw[->] (c) -- (5.2,0.3);
 \draw[<->,dashed] (4.6,1.75) to[bend right=30] node[lab,below] {$\sigma$} (5.8,1.75);
 \node[lab] at (5.2,-0.3) {(R)};
\end{tikzpicture}
\caption{The two seeds of Lemma~\ref{lem:seed}. The exchanged vertices $v,w$ are the last vertices moved by $\sigma$; their children $r_1,r_2$ or $c$ are fixed reticulations. White circles are tree nodes, grey circles are reticulations, and edges are directed downwards.}
\label{fig:seeds}
\end{figure}
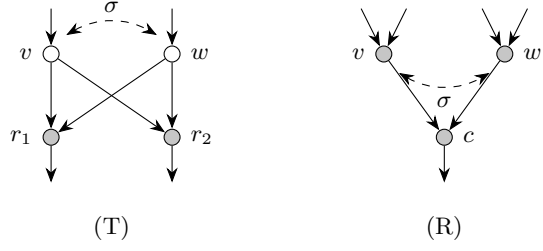

\begin{lemma}\label{lem:autbound}
Every element of $\Aut(N)$ of prime order is an involution. Consequently $\Aut(N)$ is a $2$-group.
\end{lemma}

\begin{proof}
Let $\sigma$ have prime order $p$ and let $x$ be the last moved vertex in a topological order. The orbit of $x$ has $p$ elements, since $p$ is prime. The children of $x$ come later in the order and are therefore fixed. By Lemma~\ref{lem:seams}(i) a child $y$ of $x$ has every element of the orbit of $x$ as a parent, so $p\le2$. If an odd prime divided $|\Aut(N)|$, Cauchy's theorem would provide an element of that order.
\end{proof}

\begin{lemma}\label{lem:exponent2}
If $k\le5$, every automorphism of $N\in\cP_{\ell,k}$ has order at most two, and $\Aut(N)\cong\mathbb{Z}_2^m$ for some $m$.
\end{lemma}

\begin{proof}
By Lemma~\ref{lem:autbound} it suffices to exclude an element $\sigma$ of order four. Let $Y$ be the set of vertices with $\sigma$-orbits of size four; $Y\neq\emptyset$, since otherwise $\sigma^2=\mathrm{id}$. Let $y\in Y$ be the last element of $Y$ in a topological order.

A child $c$ of $y$ is not in $Y$, and it is not fixed by $\sigma$, since a fixed child would have the four parents $y,\sigma y,\sigma^2y,\sigma^3y$. Hence $c$ has an orbit $\{c,\sigma c\}$ of size two. Applying $\sigma^2$ to the edge $(y,c)$ shows that $c$ is a reticulation with parents $y$ and $\sigma^2y$. Likewise $\sigma c$ is a reticulation with parents $\sigma y$ and $\sigma^3y$.

If $y$ is a reticulation, then $y,\sigma y,\sigma^2y,\sigma^3y,c,\sigma c$ are six distinct reticulations.

Let $y$ be a tree node with children $c_1,c_2$. The reticulations $c_1,c_2,\sigma c_1,\sigma c_2$ are distinct, because $c_1$ and $c_2$ have the parent set $\{y,\sigma^2y\}$, whereas $\sigma c_1$ and $\sigma c_2$ have the parent set $\{\sigma y,\sigma^3y\}$. None of the four is a descendant of another. Indeed, such a relation would place one of $y,\sigma y,\sigma^2y,\sigma^3y$ below a child of another of them, and applying a suitable power of $\sigma$ would close a directed cycle.

For $i=1,2$ follow a path from $c_i$ to a leaf and let $s_i$ be its first fixed vertex. The predecessor of $s_i$ is moved, so by Lemma~\ref{lem:seams}(i) the vertex $s_i$ is a reticulation whose parents are a moved vertex and its image. Being fixed, $s_1$ and $s_2$ differ from the four moved reticulations. If $s_1\neq s_2$, six reticulations occur.

Suppose $s_1=s_2=:s$, with parents $z$ and $\sigma z$. If both paths enter $s$ through the same parent, they meet before $s$. Their first common vertex $m$ is then a moved reticulation, because neither of $c_1,c_2$ is a descendant of the other. If the paths enter through different parents, say the path from $c_1$ through $z$ and the path from $c_2$ through $\sigma z$, then $\sigma z$ is a descendant of $c_2$, and, applying $\sigma$ to the path from $c_1$ to $z$, also of $\sigma c_1$. The paths from $c_2$ and from $\sigma c_1$ to $\sigma z$ consist of moved vertices and meet in a moved reticulation $m$. In both cases $m$ is a common descendant of two of the four reticulations, so it differs from all four; the same holds for $\sigma m$, and $\sigma m\neq m$. Hence $c_1,c_2,\sigma c_1,\sigma c_2,m,\sigma m,s$ are seven reticulations. In all cases at least six reticulations occur.
\end{proof}

A group of exponent two is abelian, which gives the second statement of Lemma~\ref{lem:exponent2}. The bound six is attained.

\begin{example}\label{ex:dihedral}
Consider the network of Fig.~\ref{fig:dihedral}, with two leaves and six reticulations. Its edges are $p\to x$, $p\to x'$, $x\to y_1$, $x\to y_3$, $x'\to y_2$, $x'\to y_4$, then $y_1,y_3\to c_1,c_2$ and $y_2,y_4\to c'_1,c'_2$, then $c_1,c'_1\to s_1$ and $c_2,c'_2\to s_2$, and finally $s_1$ and $s_2$ above the leaves $1$ and $2$. The permutation
\[
\sigma=(x\,x')(y_1\,y_2\,y_3\,y_4)(c_1\,c'_1)(c_2\,c'_2)
\]
maps edges onto edges; for instance $(x',y_4)\mapsto(x,y_1)$ and $(y_4,c'_2)\mapsto(y_1,c_2)$. It is an automorphism of order four. The transposition $\tau=(y_1\,y_3)$ is an automorphism as well. Since $\sigma\tau\sigma^{-1}=(y_2\,y_4)\neq\tau$, the group generated by $\sigma$ and $\tau$ is dihedral of order eight.
\end{example}

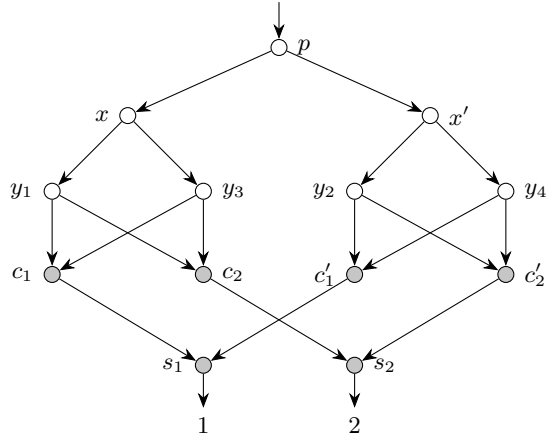
\begin{figure}[t]
\centering
\begin{tikzpicture}[>={Stealth[length=2mm]},
  tn/.style={circle,draw,fill=white,inner sep=0pt,minimum size=6pt},
  rt/.style={circle,draw,fill=gray!45,inner sep=0pt,minimum size=6pt},
  lab/.style={font=\footnotesize}]
 \node[tn,label={[lab]right:$p$}] (p) at (0,5.0) {};
 \node[tn,label={[lab]left:$x$}] (x) at (-2,4.1) {};
 \node[tn,label={[lab]right:$x'$}] (xx) at (2,4.1) {};
 \node[tn,label={[lab]left:$y_1$}] (y1) at (-3,3.1) {};
 \node[tn,label={[lab]right:$y_3$}] (y3) at (-1,3.1) {};
 \node[tn,label={[lab]left:$y_2$}] (y2) at (1,3.1) {};
 \node[tn,label={[lab]right:$y_4$}] (y4) at (3,3.1) {};
 \node[rt,label={[lab]left:$c_1$}] (c1) at (-3,2.0) {};
 \node[rt,label={[lab]right:$c_2$}] (c2) at (-1,2.0) {};
 \node[rt,label={[lab]left:$c'_1$}] (d1) at (1,2.0) {};
 \node[rt,label={[lab]right:$c'_2$}] (d2) at (3,2.0) {};
 \node[rt,label={[lab]left:$s_1$}] (s1) at (-1,0.8) {};
 \node[rt,label={[lab]right:$s_2$}] (s2) at (1,0.8) {};
 \node[lab] (l1) at (-1,0.0) {$1$};
 \node[lab] (l2) at (1,0.0) {$2$};
 \draw[->] (0,5.6) -- (p);
 \draw[->] (p) -- (x); \draw[->] (p) -- (xx);
 \draw[->] (x) -- (y1); \draw[->] (x) -- (y3);
 \draw[->] (xx) -- (y2); \draw[->] (xx) -- (y4);
 \draw[->] (y1) -- (c1); \draw[->] (y1) -- (c2);
 \draw[->] (y3) -- (c1); \draw[->] (y3) -- (c2);
 \draw[->] (y2) -- (d1); \draw[->] (y2) -- (d2);
 \draw[->] (y4) -- (d1); \draw[->] (y4) -- (d2);
 \draw[->] (c1) -- (s1); \draw[->] (d1) -- (s1);
 \draw[->] (c2) -- (s2); \draw[->] (d2) -- (s2);
 \draw[->] (s1) -- (l1); \draw[->] (s2) -- (l2);
\end{tikzpicture}
\caption{The network of Example~\ref{ex:dihedral}. White circles are tree nodes, grey circles are reticulations, and edges are directed downwards. Its automorphism group is dihedral of order eight and contains an element of order four.}
\label{fig:dihedral}
\end{figure}

The following configurations occur in the classification for $k\le3$; see Fig.~\ref{fig:gadgets}. An entry is an edge from outside the configuration into one of its vertices. An exit is an edge from one of its vertices to the outside.

\begin{definition}\label{def:gadgets}
The crossed cherry $\Gamma_1$ consists of tree nodes $p,v,w$ and reticulations $r_1,r_2$, with edges $p\to v$, $p\to w$ and $u\to r_i$ for $u\in\{v,w\}$ and $i\in\{1,2\}$. It has one entry, into $p$, and two exits, out of $r_1$ and $r_2$.

The twin reticulations $\Gamma_2$ consist of tree nodes $a,b$ and reticulations $v,w,c$, with edges $a\to v$, $a\to w$, $b\to v$, $b\to w$, $v\to c$ and $w\to c$. They have two entries, into $a$ and $b$, and one exit, out of $c$.

The propagated cross $\Gamma_3$ consists of tree nodes $p,q,q',v,w$ and reticulations $x,r_1,r_2$, with edges $p\to q$, $p\to q'$, $q\to v$, $q\to x$, $q'\to w$, $q'\to x$ and $u\to r_i$ for $u\in\{v,w\}$ and $i\in\{1,2\}$. It has one entry, into $p$, and three exits, out of $x$, $r_1$ and $r_2$.
\end{definition}

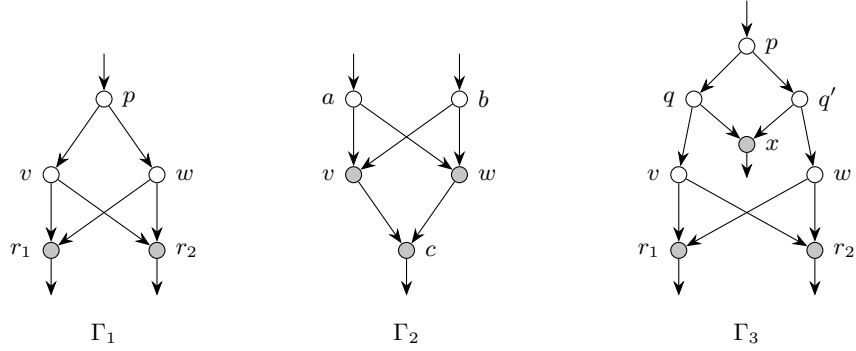
\begin{figure}[t]
\centering
\begin{tikzpicture}[>={Stealth[length=2mm]},
  tn/.style={circle,draw,fill=white,inner sep=0pt,minimum size=6pt},
  rt/.style={circle,draw,fill=gray!45,inner sep=0pt,minimum size=6pt},
  lab/.style={font=\footnotesize}]
 \node[tn,label={[lab]right:$p$}] (p) at (0,3.0) {};
 \node[tn,label={[lab]left:$v$}] (v) at (-0.7,2.0) {};
 \node[tn,label={[lab]right:$w$}] (w) at (0.7,2.0) {};
 \node[rt,label={[lab]left:$r_1$}] (r1) at (-0.7,1.0) {};
 \node[rt,label={[lab]right:$r_2$}] (r2) at (0.7,1.0) {};
 \draw[->] (0,3.6) -- (p);
 \draw[->] (p) -- (v); \draw[->] (p) -- (w);
 \draw[->] (v) -- (r1); \draw[->] (v) -- (r2);
 \draw[->] (w) -- (r1); \draw[->] (w) -- (r2);
 \draw[->] (r1) -- (-0.7,0.4); \draw[->] (r2) -- (0.7,0.4);
 \node[lab] at (0,-0.1) {$\Gamma_1$};
 \node[tn,label={[lab]left:$a$}] (a) at (3.3,3.0) {};
 \node[tn,label={[lab]right:$b$}] (b) at (4.7,3.0) {};
 \node[rt,label={[lab]left:$v$}] (v2) at (3.3,2.0) {};
 \node[rt,label={[lab]right:$w$}] (w2) at (4.7,2.0) {};
 \node[rt,label={[lab]right:$c$}] (c) at (4.0,1.0) {};
 \draw[->] (3.3,3.6) -- (a); \draw[->] (4.7,3.6) -- (b);
 \draw[->] (a) -- (v2); \draw[->] (a) -- (w2);
 \draw[->] (b) -- (v2); \draw[->] (b) -- (w2);
 \draw[->] (v2) -- (c); \draw[->] (w2) -- (c);
 \draw[->] (c) -- (4.0,0.4);
 \node[lab] at (4.0,-0.1) {$\Gamma_2$};
 \node[tn,label={[lab]right:$p$}] (p3) at (8.5,3.7) {};
 \node[tn,label={[lab]left:$q$}] (q) at (7.8,3.0) {};
 \node[tn,label={[lab]right:$q'$}] (qq) at (9.2,3.0) {};
 \node[rt,label={[lab]right:$x$}] (x) at (8.5,2.4) {};
 \node[tn,label={[lab]left:$v$}] (v3) at (7.6,2.0) {};
 \node[tn,label={[lab]right:$w$}] (w3) at (9.4,2.0) {};
 \node[rt,label={[lab]left:$r_1$}] (r13) at (7.6,1.0) {};
 \node[rt,label={[lab]right:$r_2$}] (r23) at (9.4,1.0) {};
 \draw[->] (8.5,4.3) -- (p3);
 \draw[->] (p3) -- (q); \draw[->] (p3) -- (qq);
 \draw[->] (q) -- (v3); \draw[->] (q) -- (x);
 \draw[->] (qq) -- (w3); \draw[->] (qq) -- (x);
 \draw[->] (x) -- (8.5,1.95);
 \draw[->] (v3) -- (r13); \draw[->] (v3) -- (r23);
 \draw[->] (w3) -- (r13); \draw[->] (w3) -- (r23);
 \draw[->] (r13) -- (7.6,0.4); \draw[->] (r23) -- (9.4,0.4);
 \node[lab] at (8.5,-0.1) {$\Gamma_3$};
\end{tikzpicture}
\caption{The configurations $\Gamma_1$, $\Gamma_2$ and $\Gamma_3$ of Definition~\ref{def:gadgets}. White circles are tree nodes and grey circles are reticulations. Dangling arrows are entries (above) and exits (below). The edges are directed downwards.}
\label{fig:gadgets}
\end{figure}

Each configuration admits a contraction. The configuration $\Gamma_1$ is replaced by a single tree node $u$: the entry becomes the in-edge of $u$, and the two exits become the out-edges of $u$. The configuration $\Gamma_2$ is replaced by a single reticulation $\rho$: the two entries become the in-edges of $\rho$, and the exit becomes its out-edge. The configuration $\Gamma_3$ is replaced by two tree nodes $u\to u'$: the entry becomes the in-edge of $u$, the exit out of $x$ becomes the second out-edge of $u$, and the exits out of $r_1,r_2$ become the out-edges of $u'$. In each case the inverse operation, expansion, restores the configuration. The contracted network has $2$, $2$ and $3$ fewer reticulations, respectively.

\begin{proposition}\label{prop:class}
Let $N\in\cP_{\ell,k}$ with $k\le 3$ and $\Aut(N)\neq\{\mathrm{id}\}$. Then $k\ge 2$.
If $k=2$, then $N$ contains exactly one copy of $\Gamma_1$ and $\Aut(N)\cong\mathbb{Z}_2$.
If $k=3$, then exactly one of the following holds.
(A) $N$ contains a copy of $\Gamma_2$ whose two entries leave the same tree node; then $\Aut(N)\cong\mathbb{Z}_2\times\mathbb{Z}_2$.
(B) $N$ contains a copy of $\Gamma_1$ whose two exits enter distinct vertices; then $\Aut(N)\cong\mathbb{Z}_2$.
(C) $N$ contains a copy of $\Gamma_2$ whose two entries leave distinct vertices; then $\Aut(N)\cong\mathbb{Z}_2$.
(D) $N$ contains a copy of $\Gamma_3$; then $\Aut(N)\cong\mathbb{Z}_2$.
\end{proposition}

\begin{proof}
Let $\sigma\neq\mathrm{id}$ and let $v,w$ be as in Lemma~\ref{lem:seed}. The argument counts the reticulations that $\sigma$ forces. Moved reticulations occur in pairs $\{y,\sigma(y)\}$, and seams are additional fixed reticulations.

Consider first case (T). The seams $r_1,r_2$ are two reticulations. Let $p_v$ and $p_w$ be the parents of $v$ and $w$. Since $\sigma$ maps the in-edge of $v$ onto the in-edge of $w$, one has $\sigma(p_v)=p_w$. If $p_v$ is fixed, then $p_w=p_v=:p$, and Lemma~\ref{lem:seams}(ii) shows that $p$ is a tree node with children $v,w$. The vertices $p,v,w,r_1,r_2$ then form $\Gamma_1$. If $p_v$ is moved and is a reticulation, then $p_v,p_w,r_1,r_2$ are four distinct reticulations, so $k\ge4$. If $p_v$ is a moved tree node, let $y$ be its child other than $v$; then $\sigma(y)$ is the child of $p_w$ other than $w$. Suppose $y$ is moved. A directed path from $y$ to a leaf contains a last moved vertex $y'$, and by Lemma~\ref{lem:seams}(i) its successor on the path is a seam with parent set $\{y',\sigma(y')\}$. This seam differs from $r_1$ and $r_2$, because $v$ and $w$ are not descendants of $y$; the only parent of $v$ is $p_v$, and a path from $y$ to $p_v$ would close a cycle. If $y$ or $y'$ is a reticulation, the pair $\{y',\sigma(y')\}$ adds two more reticulations. Otherwise $y'$ is a tree node, and its second child is either a further seam or a moved vertex below which a further seam occurs. In all cases $k\ge4$. Hence, for $k\le3$, the vertex $y$ is fixed and is a seam with parent set $\{p_v,p_w\}$; this is the third reticulation. Repeating the argument for the parents of $p_v$ and $p_w$ shows that they coincide in a tree node $p$, since otherwise a fourth reticulation would be forced. The resulting configuration is $\Gamma_3$.

Consider next case (R). The vertices $v,w,c$ are three reticulations, so for $k\le3$ no further reticulation exists. A parent of $v$ that is a reticulation has $v$ as its only child; if it were fixed it would also be a parent of $w$, and if it were moved it would add two reticulations. Hence the parents $a,b$ of $v$ are tree nodes, and the parents of $w$ are $\sigma(a),\sigma(b)$. Suppose a parent of $v$, say $a$, has a child $z\notin\{v,w\}$. Then $z$ is fixed or moved. A fixed $z$ is a seam by Lemma~\ref{lem:seams}(i) when $\sigma(a)\neq a$; a moved $z$ forces a seam or a reticulation pair below it as before. Either way $k\ge4$. A fixed $a$ with child $z$ is excluded as well: by Lemma~\ref{lem:seams}(ii) the children of $a$ are then $v$ and $\sigma(v)=w$. Therefore both $a$ and $b$ have the children $\{v,w\}$, and $\{a,b,v,w,c\}$ is a copy of $\Gamma_2$.

It remains to determine $\Aut(N)$ and to separate the cases. In $\Gamma_1$ the transposition $(v\,w)$ is an automorphism. In $\Gamma_2$ the transposition $(v\,w)$ is an automorphism. The transposition $(a\,b)$ is an automorphism if and only if $a$ and $b$ have the same parent; this is case (A). In case (A), the vertex set $\{p,a,b,v,w\}$, where $p$ is the common parent, is itself a copy of $\Gamma_1$ whose exits enter the common reticulation $c$. In $\Gamma_3$ the product $(q\,q')(v\,w)$ is an automorphism. For $k=2$ only case (T) with a fixed $p_v$ is possible, since the other cases force at least three reticulations. For $k=3$, every nontrivial automorphism has a seed of type (T) or (R) built from the three reticulations of $N$, and the possibilities are exactly those listed. In case (A) both seeds are present and generate $\mathbb{Z}_2\times\mathbb{Z}_2$; in cases (B), (C) and (D) only one seed is present.
\end{proof}

\begin{theorem}\label{thm:defects}
For $\ell\ge1$,
\begin{align}
D_1(\ell)&=0,\label{eq:D1}\\
D_2(\ell)&=\tfrac12(\ell-1)\,t_\ell,\label{eq:D2}\\
D_3(\ell)&=\tfrac34(2\ell-1)\,t_\ell+\tfrac12(\ell+1)\,|\cP_{\ell,1}|+\tfrac12(\ell-2)^{+}\,t_\ell,\label{eq:D3}
\end{align}
where $(\ell-2)^+=\max(\ell-2,0)$ and $|\cP_{\ell,1}|=\ell\,(2\ell)!/(2^\ell\ell!)-2^{\ell-1}\ell!$ \cite{Zhang19}.
\end{theorem}

\begin{proof}
Equation \eqref{eq:D1} follows from Proposition~\ref{prop:class}. For the remaining identities, each class of Proposition~\ref{prop:class} is put in bijection with marked networks of lower reticulation number by contraction. Each network then contributes $1-|\Aut(N)|^{-1}$, which equals $\tfrac12$ for $\mathbb{Z}_2$ and $\tfrac34$ for $\mathbb{Z}_2\times\mathbb{Z}_2$.

For $k=2$, contracting $\Gamma_1$ yields a tree on $[\ell]$ with a marked tree node. Conversely, expanding a tree node of a tree yields a network of $\cP_{\ell,2}$ containing $\Gamma_1$. A tree has $\ell-1$ tree nodes, so there are $(\ell-1)t_\ell$ such networks, which gives \eqref{eq:D2}.

For $k=3$ the four classes are counted as follows. In class (A) the vertices $p,a,b,v,w,c$ have one entry and one exit. Contracting them to a single edge yields a tree on $[\ell]$ with a marked edge. There are $(2\ell-1)t_\ell$ such pairs, each weighted by $\tfrac34$. In class (B), contracting $\Gamma_1$ yields a network $N'\in\cP_{\ell,1}$ with a marked tree node $u$. No parallel edge arises, because the two exits enter distinct vertices. Every one of the $\ell$ tree nodes of every $N'$ can be expanded, so this class has $\ell|\cP_{\ell,1}|$ elements. In class (C), contracting $\Gamma_2$ yields $N'\in\cP_{\ell,1}$ together with its unique reticulation; the entries leave distinct vertices, so no parallel edge arises, and the class has $|\cP_{\ell,1}|$ elements. In class (D), contracting $\Gamma_3$ yields a tree with a marked tree node $u'$ whose parent $u$ is a tree node. Every tree node except the child of the root qualifies, so the class has $(\ell-2)t_\ell$ elements for $\ell\ge2$ and none for $\ell=1$. Classes (B), (C) and (D) are weighted by $\tfrac12$, and summing gives \eqref{eq:D3}.
\end{proof}

The first values are $D_2(\ell)=0,\tfrac12,3,\tfrac{45}{2}$ and $D_3(\ell)=\tfrac34,\tfrac{21}{4},\tfrac{219}{4},\tfrac{2655}{4},\tfrac{37125}{4}$ for $\ell=1,2,\dots$.

\section{The layer-peeling recursion}\label{sec:recursion}

\begin{definition}\label{def:arooted}
Let $a\ge1$. An $a$-rooted structure is a finite acyclic directed multigraph with the following properties. It has $a$ roots, labelled by $[a]$, each of in-degree $0$ and out-degree $1$. It has labelled leaves and the tree nodes and reticulations of Definition~\ref{def:network}. Every vertex is reachable from a root. Parallel edges occur only as parallel pairs, that is, as a tree node both of whose out-edges enter the same reticulation. The weight of a structure $S$ is $\omega(S)=1/|\Aut_e(S)|$, where $\Aut_e(S)$ is the group of incidence-preserving permutations of vertices and edges fixing every root and every leaf. In particular, each parallel pair contributes the transposition of its two edges. We write $Q_a(\ell,m;p)$ for the sum of $\omega(S)$ over the structures with $\ell$ leaves, $m$ reticulations and exactly $p$ parallel pairs, and set $Q_0(\ell,m;p)=[\ell=m=p=0]$.
\end{definition}

A $1$-rooted structure without parallel pairs is a network, and its root edge is the root edge of Definition~\ref{def:network}. Consequently
\begin{equation}\label{eq:WQ}
W_k(\ell)=Q_1(\ell,k;0).
\end{equation}

Counting degrees as in Section~\ref{sec:weighted}, an $a$-rooted structure with $\ell$ leaves and $m$ reticulations has $\ell+m-a$ tree nodes and
\begin{equation}\label{eq:edges}
E_a(\ell,m)=2\ell+3m-a
\end{equation}
edges, including the root edges.

Removing every edge that enters a reticulation decomposes a structure $S$ into trees rooted at the roots of $S$ and at its reticulations. The layer $\Lambda(S)$ is the union of the $a$ trees rooted at the roots of $S$. A reticulation of $S$ is closed if both of its parents lie in $\Lambda(S)$, open if exactly one does, and internal otherwise. For integers $E\ge0$ and $b\ge0$ we write $\rise{E}{b}=E(E+1)\cdots(E+b-1)$, with $\rise{E}{0}=1$.

The recursion for $Q_a$ rests on a decomposition of a structure $S$ into three pieces, illustrated in Fig.~\ref{fig:layer}. The first piece is the layer forest. It is formed by the trees of $\Lambda(S)$, in which every edge leaving the layer is cut and its head is replaced by a new leaf, called a slot. The second piece is the remainder, the structure left after deleting $\Lambda(S)$, whose roots are the closed reticulations. The third piece is the placement, which records the positions on the edges of the remainder at which the open reticulations are reinserted.

The indices of the recursion record these choices. The index $j$ is the number of leaves of $S$ in the layer, and $a'$ and $b'$ are the numbers of closed and open reticulations. A parallel pair of $S$ either arises inside the layer, as a cherry formed by the two slots of a closed reticulation, or is a parallel pair of the remainder that receives no reinserted reticulation. Accordingly, $p_1$ is the number of parallel pairs of the first kind, $p_2$ is the number of parallel pairs of the remainder, and $q=p-p_1$ of these must remain unbroken. The layer factor $L_a$ counts layer forests with exactly $p_1$ such cherries. The placement factor $\Pi$ counts the placements of $b'$ points on the $E'$ edges of the remainder that leave exactly $q$ of its parallel pairs unbroken. In the example of Fig.~\ref{fig:layer} one has $j=1$, $a'=1$, $b'=1$ and $m'=0$: the reticulation $x$ is closed, the reticulation $y$ is open, and the remainder is a single tree carrying one point.

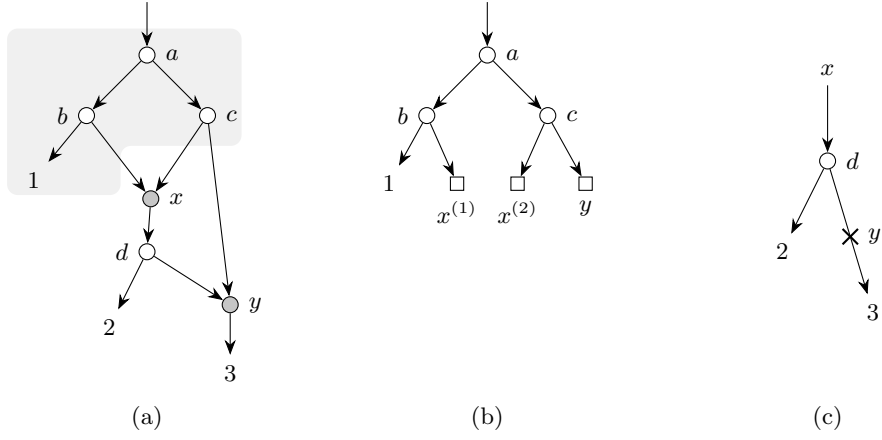
\begin{figure}[t]
\centering
\begin{tikzpicture}[>={Stealth[length=2mm]},
  tn/.style={circle,draw,fill=white,inner sep=0pt,minimum size=6pt},
  rt/.style={circle,draw,fill=gray!45,inner sep=0pt,minimum size=6pt},
  sl/.style={rectangle,draw,fill=white,inner sep=0pt,minimum size=5pt},
  pt/.style={cross out,draw,inner sep=0pt,minimum size=5pt,thick},
  lab/.style={font=\footnotesize}]
 \fill[gray!12,rounded corners] (-1.85,4.15) -- (1.25,4.15) -- (1.25,2.6) -- (-0.35,2.6) -- (-0.35,1.95) -- (-1.85,1.95) -- cycle;
 \node[tn,label={[lab]right:$a$}] (a) at (0,3.8) {};
 \node[tn,label={[lab]left:$b$}] (b) at (-0.8,3.0) {};
 \node[tn,label={[lab]right:$c$}] (c) at (0.8,3.0) {};
 \node[lab] (l1) at (-1.5,2.15) {$1$};
 \node[rt,label={[lab]right:$x$}] (x) at (0.05,1.9) {};
 \node[rt,label={[lab]right:$y$}] (y) at (1.1,0.5) {};
 \node[tn,label={[lab]left:$d$}] (d) at (0,1.2) {};
 \node[lab] (l2) at (-0.5,0.2) {$2$};
 \node[lab] (l3) at (1.1,-0.4) {$3$};
 \draw[->] (0,4.5) -- (a);
 \draw[->] (a) -- (b); \draw[->] (a) -- (c);
 \draw[->] (b) -- (l1); \draw[->] (b) -- (x);
 \draw[->] (c) -- (x); \draw[->] (c) -- (y);
 \draw[->] (x) -- (d); \draw[->] (d) -- (l2); \draw[->] (d) -- (y);
 \draw[->] (y) -- (l3);
 \node[lab] at (0,-1.0) {(a)};
 \node[tn,label={[lab]right:$a$}] (A) at (4.5,3.8) {};
 \node[tn,label={[lab]left:$b$}] (B) at (3.7,3.0) {};
 \node[tn,label={[lab]right:$c$}] (C) at (5.3,3.0) {};
 \node[lab] (L1) at (3.2,2.1) {$1$};
 \node[sl,label={[lab]below:$x^{(1)}$}] (X1) at (4.1,2.1) {};
 \node[sl,label={[lab]below:$x^{(2)}$}] (X2) at (4.9,2.1) {};
 \node[sl,label={[lab]below:$y$}] (Y) at (5.8,2.1) {};
 \draw[->] (4.5,4.5) -- (A);
 \draw[->] (A) -- (B); \draw[->] (A) -- (C);
 \draw[->] (B) -- (L1); \draw[->] (B) -- (X1);
 \draw[->] (C) -- (X2); \draw[->] (C) -- (Y);
 \node[lab] at (4.5,-1.0) {(b)};
 \node[tn,label={[lab]right:$d$}] (D) at (9.0,2.4) {};
 \node[lab] (L2) at (8.4,1.2) {$2$};
 \node[lab] (L3) at (9.6,0.4) {$3$};
 \draw[->] (9.0,3.4) node[above,lab] {$x$} -- (D);
 \draw[->] (D) -- (L2); \draw[->] (D) -- (L3);
 \node[pt,label={[lab]right:$y$}] at (9.3,1.4) {};
 \node[lab] at (9.0,-1.0) {(c)};
\end{tikzpicture}
\caption{Layer peeling. (a) A network in $\cP_{3,2}$; the shaded region contains the layer, the tree component of the root. (b) The layer forest: one tree with the leaf $1$, the two slots $x^{(1)},x^{(2)}$ of the closed reticulation $x$, and the slot of the open reticulation $y$. (c) The remainder: a $1$-rooted structure with root $x$ and leaves $2,3$, carrying the point $y$ on the edge into the leaf $3$. Here $j=1$, $a'=1$, $b'=1$ and $m'=0$.}
\label{fig:layer}
\end{figure}

\begin{theorem}\label{thm:recursion}
For $a\ge1$,
\begin{multline}\label{eq:recursion}
Q_a(\ell,m;p)=\sum_{j=0}^{\ell}\binom{\ell}{j}\sum_{\substack{a',b'\ge0\\ a'+b'\le m}}\frac{1}{a'!\,b'!}
\sum_{p_1=0}^{\min(p,a')}L_a(j,a',b';p_1)\\
\times\sum_{p_2\ge p-p_1}Q_{a'}(\ell-j,m';p_2)\,\Pi(E',b';p_2,p-p_1),
\end{multline}
where $m'=m-a'-b'$, $E'=E_{a'}(\ell-j,m')$, the terms with $a'=0$ are understood with $Q_0$,
\begin{equation}\label{eq:layer}
L_a(j,a',b';p_1)=\frac{a!}{2^{a'}}\sum_{s=p_1}^{a'}(-1)^{s-p_1}\binom{s}{p_1}\binom{a'}{s}\,F_{j+2a'+b'-s,\,a},
\end{equation}
and
\begin{equation}\label{eq:placement}
\Pi(E,b;p_2,q)=\binom{p_2}{q}\sum_{i=0}^{p_2-q}(-1)^i\binom{p_2-q}{i}\,\rise{(E-2q-2i)}{b}.
\end{equation}
\end{theorem}

\begin{proof}
Let $S$ be an $a$-rooted structure. Let $j$ be the number of its leaves in $\Lambda(S)$, $a'$ the number of closed reticulations and $b'$ the number of open ones. The proof decomposes $S$ into three pieces of data and then accounts for the weights.

The first piece is the layer forest. Replace the head of every edge that leaves $\Lambda(S)$ by a new leaf, called a slot. The result is a forest of $a$ trees with labelled roots. Its leaves are the $j$ leaves of $S$ in the layer, $2a'$ slots for the closed reticulations and $b'$ slots for the open ones. Label the slots auxiliarily: the closed reticulations by $[a']$, the two slots of each closed reticulation by an order, and the open reticulations by $[b']$.

The second piece is the remainder. Delete $\Lambda(S)$. Each closed reticulation becomes a root, labelled by its auxiliary label. Each open reticulation now has in-degree $1$ and out-degree $1$; it is suppressed, and its position is recorded as a labelled point on the resulting edge. Several points on one edge are linearly ordered. The result is an $a'$-rooted structure $R$ with $\ell-j$ leaves and $m'$ reticulations, together with $b'$ labelled points on its edges. Its edge count $E'$ is given by \eqref{eq:edges}.

The third piece records the parallel pairs. A parallel pair of $S$ is either a cherry of the layer forest formed by the two slots of one closed reticulation, or a parallel pair of $R$ that carries no point. A parallel pair of $R$ that carries at least one point is not a parallel pair of $S$.

The construction is reversible. Given a forest with labelled roots and labelled slots, an $a'$-rooted structure with labelled roots, and a placement of $b'$ labelled points on its edges, one recovers $S$ uniquely. The two slots of closed reticulation $i$ are merged into a reticulation that is identified with root $i$ of the remainder. Each point is replaced by a reticulation whose second parent is the parent of the corresponding open slot.

Placing $b$ labelled points successively on a structure with $E$ edges, where each point subdivides the edge it lands on, can be done in $\rise{E}{b}$ ways; excluding a set of $f$ edges gives $\rise{(E-f)}{b}$. Equation \eqref{eq:placement} is the inclusion–exclusion count of placements in which a prescribed set of $q$ of the $p_2$ parallel pairs of $R$ receives no point and each of the other $p_2-q$ pairs receives at least one.

Forests with $a$ labelled roots on $n$ labelled leaves number $a!\,F_{n,a}$. Prescribing that $s$ given slot pairs form cherries amounts to contracting each such cherry to a single leaf, leaving $a!\,F_{n-s,a}$ forests. Inclusion–exclusion gives the number with exactly $p_1$ cherry pairs. The factor $2^{-a'}$ removes the auxiliary order of the two slots in each pair. When the two slots of a pair form a cherry, exchanging them does not change the labelled forest; the corresponding factor $\tfrac12$ is the weight of the parallel pair in $\omega$. This gives \eqref{eq:layer}.

For the weights, let $N_S$ be the number of internal vertices of $S$, that is, of vertices other than roots and leaves. The group $\Aut_e(S)$ acts freely on the set of rigid decorations of $S$, each consisting of a bijective labelling of the internal vertices by $[N_S]$ together with an order of the two edges of every parallel pair. By the orbit--stabiliser theorem the number of rigid decorations is $N_S!\,2^{p}/|\Aut_e(S)|$, so that $\omega(S)$ equals this number divided by $N_S!\,2^{p}$. Applied to rigidly decorated structures, the decomposition above is a bijection onto rigidly decorated data, namely layer forest, remainder and placement, supplemented by an auxiliary labelling of the closed reticulations by $[a']$, of the open reticulations by $[b']$ and an order of the two slots of every closed reticulation. The internal labels are distributed between the layer and the remainder; this is the labelled product of species and contributes no factor once both sides are divided by their factorials. The auxiliary data are acted on freely by $S_{a'}\times S_{b'}\times\mathbb{Z}_2^{a'}$. The first two factors produce $1/(a'!\,b'!)$. The third produces $2^{-a'}$ in \eqref{eq:layer}, except on cherry pairs, where the order of the two slots coincides with the order of the two edges of the resulting parallel pair and the corresponding factor $\tfrac12$ is the one carried by $\omega$. The binomial $\binom{\ell}{j}$ distributes the leaf labels between the layer and the remainder. Summing over $j,a',b',p_1$ and $p_2$ gives \eqref{eq:recursion}.

Every term with $a'\ge1$ refers to a remainder with fewer reticulations than $S$. Every term with $a'=0$ forces an empty remainder. The recursion therefore terminates.
\end{proof}

Write $u=\sqrt{1-2x}$, so that $x=\frac12(1-u^2)$, and let $T(x)=1-u$ be the exponential generating function $\sum_n t_nx^n/n!$ of rooted binary trees. Let $D=\frac{d}{dx}$ and $\theta=xD$. From $du/dx=-1/u$ one obtains, for every integer $n$,
\begin{equation}\label{eq:uops}
D\,u^n=-n\,u^{n-2},\qquad \theta\,u^n=\tfrac n2\bigl(u^n-u^{n-2}\bigr),
\end{equation}
so that $D$ and $\theta$ act on the ring $\mathbb{Q}[u,u^{-1}]$ of Laurent polynomials in $u$.

\begin{theorem}\label{thm:closure}
For all $a,m,p$, the exponential generating function
\[
\widehat Q_{a,m,p}(x)=\sum_{\ell\ge0}Q_a(\ell,m;p)\,\frac{x^\ell}{\ell!}
\]
is a Laurent polynomial in $u$ with rational coefficients. It is obtained from \eqref{eq:recursion} by three substitutions:
\begin{itemize}
\item the convolution $\sum_j\binom{\ell}{j}$ becomes the product of the generating functions of the layer factor and of the remainder term;
\item the forest number $F_{j+c,a}$ in the layer factor becomes $D^{c}\bigl(T^a/a!\bigr)$;
\item the placement factor, a polynomial $g$ in the number $r=\ell-j$ of leaves of the remainder, becomes the operator $g(\theta)$ applied to $\widehat Q_{a',m',p_2}$.
\end{itemize}
Moreover, for every integer $n$ and $\ell\ge0$,
\begin{equation}\label{eq:uextract}
\ell!\,[x^\ell]\,u^n=\prod_{i=0}^{\ell-1}(2i-n).
\end{equation}
\end{theorem}

\begin{proof}
The first substitution is the product rule for exponential generating functions. For the second, $\sum_nF_{n,a}x^n/n!=T(x)^a/a!$, and $\sum_jF_{j+c,a}x^j/j!$ is its $c$-th derivative. For the third, $\theta$ multiplies the coefficient of $x^r/r!$ by $r$. The edge number of the remainder is $E'=2r+3m'-a'$, so each rising factorial in \eqref{eq:placement} is a polynomial in $r$. Its polynomial extension agrees with the value used in \eqref{eq:placement} whenever the remainder count is nonzero. Indeed, a remainder with $p_2$ parallel pairs has at least $2p_2$ edges, so the argument $E'-2q-2i$ is nonnegative for $i\le p_2-q$. By \eqref{eq:uops} the ring $\mathbb{Q}[u,u^{-1}]$ is closed under $D$ and $\theta$, and $T\in\mathbb{Q}[u]$. Induction on $m$ with Theorem~\ref{thm:recursion} proves the first statement. Finally, $u^n=(1-2x)^{n/2}$ and $[x^\ell](1-2x)^{\alpha}=(-2)^\ell\binom{\alpha}{\ell}$ give \eqref{eq:uextract}.
\end{proof}

\begin{corollary}\label{cor:shape}
Let $f(\ell)=\ell!\,[x^\ell]\,G(x)$ for a Laurent polynomial $G=\sum_ng_nu^n$ with some odd $n$ such that $g_n\ne0$, and let $N$ be the largest such odd $n$. There are polynomials $A,B$ with rational coefficients such that
\[
f(\ell)=A(\ell)\,(2\ell-N-2)!!-B(\ell)\,2^\ell\ell!\qquad\text{for }\ell>\ell_0,
\]
where $\ell_0$ is the larger of $\frac{N-1}{2}$ and half the largest even $n\ge0$ with $g_n\ne0$. For $\ell\le\ell_0$ finitely many corrections from the powers $u^{2s}$ with $s\ge0$ are added.
\end{corollary}

\begin{proof}
By \eqref{eq:uextract}, the odd power $u^{N-2d}$ contributes
\[
g_{N-2d}\prod_{i=0}^{\ell-1}(2i-N+2d)=g_{N-2d}\,\frac{\prod_{j=0}^{d-1}(2\ell+2j-N)}{\prod_{i=0}^{d-1}(2i-N)}\,\prod_{i=0}^{\ell-1}(2i-N).
\]
Here $\prod_{i=0}^{\ell-1}(2i-N)=\kappa\,(2\ell-N-2)!!$ for $\ell\ge\frac{N+1}{2}$, with $\kappa=1/(-N-2)!!$ if $N<0$ and $\kappa=(-1)^{(N+1)/2}N!!$ if $N>0$; we use $(-1)!!=1$. A power $u^{-2s}$ with $s\ge1$ contributes $g_{-2s}\,2^\ell(\ell+s-1)!/(s-1)!=g_{-2s}\,2^\ell\ell!\binom{\ell+s-1}{s-1}$. A power $u^{2s}$ with $s\ge0$ contributes only for $\ell\le s$.
\end{proof}

By Theorem~\ref{thm:closure} and \eqref{eq:WQ}, Corollary~\ref{cor:shape} applies to $W_k$ for every $k$. The finitely many exceptional values of $\ell$ come from the polynomial part of the generating function. This is the origin of the separate statements $|\cP_{1,3}|=9$ and $|\cP_{1,4}|=109$ in \cite{Yu26,YuZhang26}.

\section{Comparison with the published counts for at most three reticulations}\label{sec:verification}

This section compares the weighted counts given by Theorem~\ref{thm:recursion} with the generating functions of \cite{Mansouri22} and, after addition of the defects of Theorem~\ref{thm:defects}, with the published totals. As a worked example, take $W_2(2)=Q_1(2,2;0)$. With $p=0$ only $p_1=0$ and $p_2=q=0$ occur. Three families of terms contribute; all others vanish because they require at least two leaves in a remainder with two roots, or a parallel pair in the remainder that no point breaks.

The terms with $(a',b')=(2,0)$ require $j=0$. The layer factor is $L_1(0,2,0;0)=\tfrac14(t_4-2t_3+t_2)=\tfrac52$, and the remainder is a forest of two trees with labelled roots on two leaves, $Q_2(2,0;0)=2$. With the factor $1/2!$ this family contributes $\tfrac52$.

The terms with $(a',b')=(1,1)$ involve a single tree as remainder. For $j=0$ the layer factor is $L_1(0,1,1;0)=\tfrac12(t_3-t_2)=1$, the remainder count is $t_2=1$, and the placement count is $E'=3$, giving $3$. For $j=1$ one has $\binom21L_1(1,1,1;0)\,t_1\,E'=2\cdot6\cdot1\cdot1=12$. This family contributes $15$.

The terms with $(a',b')=(1,0)$ and one internal reticulation vanish, since $L_1(0,1,0;0)=0$.

Hence $W_2(2)=\tfrac52+15=\tfrac{35}{2}$, and \eqref{eq:split} with $D_2(2)=\tfrac12$ gives $|\cP_{2,2}|=18$. The same hand computation was carried out for the other entries of Table~\ref{tab:low}. In every case the recursion and the generating functions \eqref{eq:M2}--\eqref{eq:M3} give the same weighted count, and adding the defect of Theorem~\ref{thm:defects} reproduces the published totals.

\begin{table}[t]
\centering
\begin{tabular}{cccccc}
\toprule
$k$ & $\ell$ & $W_k(\ell)$ by \eqref{eq:recursion} & $W_k(\ell)$ by \eqref{eq:WfromG} & $D_k(\ell)$ & $|\cP_{\ell,k}|$ \cite{YuZhang26}\\
\midrule
1 & 2 & $2$ & $2$ & $0$ & $2$\\
1 & 3 & $21$ & $21$ & $0$ & $21$\\
2 & 1 & $1$ & $1$ & $0$ & $1$\\
2 & 2 & $35/2$ & $35/2$ & $1/2$ & $18$\\
2 & 3 & $276$ & $276$ & $3$ & $279$\\
3 & 1 & $33/4$ & $33/4$ & $3/4$ & $9$\\
\bottomrule
\end{tabular}
\caption{Weighted counts obtained by hand from the recursion of Theorem~\ref{thm:recursion} and from the generating functions of \cite{Mansouri22}, defects from Theorem~\ref{thm:defects}, and the published totals.}
\label{tab:low}
\end{table}

For $k=3$ and larger $\ell$, Table~\ref{tab:k3} compares three quantities: the weighted count obtained from the recursion, the weighted count obtained from \eqref{eq:WfromG} with the published polynomials \eqref{eq:M3}, and the totals of \cite{Yu26}, which are reproduced in \cite[Table~1]{YuZhang26}. The recursion plus the defect \eqref{eq:D3} reproduces the published totals exactly for every $\ell\le12$. The generating function with the published numerator agrees for $\ell\le2$ only. For $\ell=3,4,5$ the differences are $6$, $216$ and $5940$, and for every $\ell\le12$ they equal
\begin{equation}\label{eq:missing}
\Delta(\ell)=\ell!\,[x^{\ell+2}]\,\frac{x^5}{(1-2x)^{9/2}}.
\end{equation}
The published numerator therefore carries a coefficient error. Replacing $a_3$ by $a_3+x^5(1-2x)$, that is,
\begin{equation}\label{eq:a3corrected}
a_3(x)=x^6+3x^5+4x^4+2x^3+\tfrac{69}{4}x^2,
\end{equation}
makes \eqref{eq:WfromG} coincide with the recursion for every $\ell\le12$. Within the form \eqref{eq:Mform} with the degree bounds $\deg a_3\le6$ and $\deg b_3\le5$ of the published derivation, and with the three vanishing coefficients imposed by $\ell\ge1$, ten coefficients remain free; agreement on twelve consecutive values determines them, and \eqref{eq:a3corrected} is the correct numerator. The published leaf-labelled value $G'_{3,1}=51$ in \cite[Eq.~(30)]{Mansouri22} is inconsistent with both generating functions, which give $W_3(1)=\tfrac{33}{4}$ and hence $|\cP_{1,3}|=9$. The agreement of the recursion with \cite{Yu26} for all $\ell\le12$ is an independent confirmation of the three-reticulation formula, by a method that shares no step with the component-graph derivation.

\begin{table}[t]
\centering
\begin{tabular}{crrrr}
\toprule
$\ell$ & $W_3(\ell)$ by \eqref{eq:recursion} & $W_3(\ell)$ from \eqref{eq:M3} & $D_3(\ell)$ & $|\cP_{\ell,3}|$ \cite{Yu26}\\
\midrule
1 & $33/4$ & $33/4$ & $3/4$ & $9$\\
2 & $879/4$ & $879/4$ & $21/4$ & $225$\\
3 & $19701/4$ & $19677/4$ & $219/4$ & $4980$\\
4 & $438165/4$ & $437301/4$ & $2655/4$ & $110205$\\
5 & $10126695/4$ & $10102935/4$ & $37125/4$ & $2540955$\\
\bottomrule
\end{tabular}
\caption{Three reticulations. The recursion plus the defect reproduces the published totals; the generating function with the published numerator falls short by $\Delta(\ell)$ of \eqref{eq:missing}. The agreement extends to $\ell\le12$.}
\label{tab:k3}
\end{table}

\section{Asymptotic enumeration for every number of reticulations}\label{sec:asympt}

This section derives statements for every number of reticulations, uniform in $k$ over explicit ranges. Write $TC_{\ell,k}$ for the number of tree-child networks in $\cP_{\ell,k}$ and $c_{n,k}$ for the word counts of \cite{PonsBatle21}. By the Pons--Batle identity, proved in \cite{Lin26},
\begin{equation}\label{eq:PB}
TC_{\ell,k}=\frac{\ell!}{(\ell-k)!}\,c_{\ell-1,k},\qquad c_{n,k}=c_{n,k-1}+(2n+k-1)\,c_{n-1,k},\qquad c_{n,0}=(2n-1)!!,
\end{equation}
with $c_{n,k}=0$ for $k>n$. We set $a_k=2^{k-1}/k!$. The letter $C$ denotes absolute constants, which may change from line to line.

\subsection{Structures, insertion and removal}

\begin{definition}\label{def:Pstar}
Let $\cP^*_{\ell,j}$ be the set of structures satisfying Definition~\ref{def:network} with $j$ reticulations, except that parallel pairs are allowed: a tree node may send both of its out-edges to the same reticulation. Such a structure has $E_j=2\ell+3j-1$ edges. A violation is a non-leaf vertex all of whose children are reticulations; the tail of a parallel pair is a violation. The tree-child networks are the elements of $\cP^*_{\ell,j}$ without violations.
\end{definition}

For edges $e_t=(a,b)$ and $e_h=(c,d)$ of $S\in\cP^*_{\ell,j-1}$ with no directed path from $d$ to $a$, the insertion $I(S,e_t,e_h)$ subdivides $e_t$ by a new tree node $t$, subdivides $e_h$ by a new reticulation $h$, and adds the edge $(t,h)$. When $e_t=e_h$ the node $t$ is placed above $h$, which creates a parallel pair. Conversely, for an edge $f=(t,h)$ of $S\in\cP^*_{\ell,j}$ with $t$ a tree node and $h$ a reticulation, the removal $R(S,f)$ deletes $f$ and suppresses $t$ and $h$, which now have in-degree and out-degree one.

\begin{lemma}\label{lem:removal}
Let $S\in\cP^*_{\ell,j}$ with $j\ge1$, and let $h$ be a reticulation that is first in a topological order among the reticulations. Then both parents of $h$ are tree nodes. For each in-edge $f$ of $h$, $R(S,f)\in\cP^*_{\ell,j-1}$ and $S=I(R(S,f),e_t,e_h)$ for edges $e_t,e_h$ of $R(S,f)$.
\end{lemma}

\begin{proof}
A parent of $h$ lies above $h$ and is not a reticulation, by the choice of $h$. It is not the root either, because the root has a single child, which would then be $h$, and the second parent of $h$ would lie below $h$. After the deletion of $f$, the vertices $t$ and $h$ have in-degree and out-degree one. Suppressing $t$ joins its parent $a$ to its other child $b$. If $(a,b)$ already exists, $b$ has two edges from $a$; then $b$ is a reticulation and $a$, which had the two children $t$ and $b$, is a tree node, so the new double edge is a parallel pair. The same holds for the suppression of $h$. Inserting at the two edges created by the suppressions restores $S$.
\end{proof}

\subsection{Tree-child counts}

\begin{lemma}\label{lem:exactTC}
Define $\alpha_0=1$ and, for $k\ge1$, $\alpha_k$ by
\begin{equation}\label{eq:alphadef}
\sum_{i=0}^{k}\binom{k}{i}\alpha_i\,(4k-i-3)!!=0 .
\end{equation}
Then, for all $k\ge0$ and $n\ge k-1$,
\begin{equation}\label{eq:cexact}
c_{n,k}=\frac{1}{k!}\sum_{i=0}^{k}\binom{k}{i}\alpha_i\,(2n+2k-i-1)!!,
\end{equation}
and hence, for $\ell\ge k$,
\begin{equation}\label{eq:TCexact}
TC_{\ell,k}=\binom{\ell}{k}\sum_{i=0}^{k}\binom{k}{i}\alpha_i\,(2\ell+2k-i-3)!! .
\end{equation}
The first values are $\alpha_0,\dots,\alpha_4=1,-1,\frac13,\frac{17}{8},-\frac{283}{63}$.
\end{lemma}

\begin{proof}
Write $\hat c_{n,k}$ for the right-hand side of \eqref{eq:cexact}. For each term,
\[
(2n+2k-i-1)!!-(2n+k-1)(2n+2k-i-3)!!=(k-i)\,(2n+2k-i-3)!!,
\]
and $\binom{k}{i}(k-i)/k!=\binom{k-1}{i}/(k-1)!$. Hence $\hat c_{n,k}-(2n+k-1)\hat c_{n-1,k}=\hat c_{n,k-1}$ for all $n$, with the same constants $\alpha_i$. Moreover, $\hat c_{n,0}=(2n-1)!!$, and $\hat c_{k-1,k}=0$ by \eqref{eq:alphadef}. The numbers $c_{n,k}$ satisfy the same recursion for $n\ge k$, with $c_{n,0}=(2n-1)!!$ and $c_{k-1,k}=0$. Induction on $k$, and for fixed $k$ on $n\ge k-1$, gives $c_{n,k}=\hat c_{n,k}$. Equation \eqref{eq:TCexact} follows from \eqref{eq:PB} with $n=\ell-1$. The values of $\alpha_i$ follow from \eqref{eq:alphadef}; for instance $\alpha_3$ solves $945-1152+105+48\alpha_3=0$.
\end{proof}

Formula \eqref{eq:TCexact} is equivalent to the semi-closed form of Lin et al.\ \cite{Lin26}. Its truncations are the approximants of Yu and Zhang \cite[App.~A]{YZasym}, which are exact for $k\le4$; the constants $\alpha_i$ do not depend on $k$. The terms with even $i$ are $(2\ell-1)!!$ times a polynomial in $\ell$, and those with odd $i$ are $2^{\ell}\ell!$ times a polynomial. This gives the two-sector form of the tree-child count for every $k$ without appeal to generating functions.

\begin{corollary}\label{cor:sectors}
For $k\ge1$ and $\ell\ge k$, $TC_{\ell,k}=\ell^{\underline k}\bigl[(2\ell-1)!!\,p_k(\ell-1)-2^{\ell-1}(\ell-1)!\,q_k(\ell-1)\bigr]$, with polynomials
\[
p_k(n)=a_k\Bigl(n^{k-1}+\tfrac{(k-1)(7k+6)}{12}\,n^{k-2}+O(n^{k-3})\Bigr),\qquad q_k(n)=k\,a_k\,n^{k-1}+O(n^{k-2}).
\]
\end{corollary}

\begin{proof}
The term $i=0$ of \eqref{eq:cexact} is $\frac1{k!}(2n+1)!!\prod_{j=1}^{k-1}(2n+2j+1)$, whose two leading coefficients in $n$ are $2^{k-1}/k!$ and $2^{k-2}\sum_{j=1}^{k-1}(2j+1)/k!=2^{k-2}(k-1)(k+1)/k!$. The term $i=2$ contributes $\frac{1}{k!}\binom k2\frac13(2n+1)!!\prod_{j=1}^{k-2}(2n+2j+1)$, of leading coefficient $2^{k-2}k(k-1)/(6\,k!)$ in degree $k-2$. Adding, the second coefficient of $p_k$ is $\frac{2^{k-2}}{k!}(k-1)\bigl(k+1+\frac k6\bigr)=a_k\frac{(k-1)(7k+6)}{12}$. The leading term of $q_k$ comes from $i=1$, namely $\frac{k}{k!}(2n)!!\prod_{j=1}^{k-1}(2n+2j)$.
\end{proof}

With $(2\ell-1)!!=\sqrt2(2\ell/e)^\ell(1-\frac1{24\ell}+O(\ell^{-2}))$ and $2^\ell\ell!/(2\ell-1)!!=\sqrt{\pi\ell}(1+O(\ell^{-1}))$, Corollary~\ref{cor:sectors} gives, for fixed $k$,
\begin{equation}\label{eq:TCexp}
TC_{\ell,k}=\frac{2^{k-1}\sqrt2}{k!}\Bigl(\frac2e\Bigr)^{\ell}\ell^{\ell+2k-1}\Bigl(1-\frac{k\sqrt\pi}{2}\ell^{-1/2}+\frac{2k^2-14k+11}{24}\ell^{-1}+O(\ell^{-3/2})\Bigr).
\end{equation}
Yu and Zhang \cite[Prop.~4.1]{YZasym} prove the equivalent expansion uniformly for $k=o(\sqrt\ell)$ and to higher order. The first-order constant is due to Fuchs, Huang and Yu \cite{FHY}. The uniform statements of this section use two estimates of Yu and Zhang \cite[Prop.~B1 and Eq.~(E15)]{YZasym}, which are its only external input. For fixed $k$ both estimates follow from Corollary~\ref{cor:sectors}, so the fixed-$k$ statements are self-contained. The first estimate is that, uniformly for $1\le j=o(\sqrt\ell)$,
\begin{equation}\label{eq:TCratio}
\frac{TC_{\ell,j-1}}{TC_{\ell,j}}=\frac{j}{2\ell(\ell-j+1)}\bigl(1+O(\ell^{-1/2})\bigr).
\end{equation}
The second is that there is $c_0>0$ such that $TC_{\ell,j-1}/TC_{\ell,j}\le Cj/\ell^2$ for $1\le j\le c_0\sqrt\ell$.

\subsection{Violations and the comparison with tree-child networks}

\begin{lemma}\label{lem:pinning}
Let $S\in\cP^*_{\ell,j}$ have a violation, and let $v$ be its first violation in a topological order. There is an edge $f$ at $v$ such that $S=I(R(S,f),e_t,e_h)$, where one of $e_t,e_h$ is an in-edge of a reticulation of $R(S,f)$ or $e_t=e_h$. We call this position pinned; given $R(S,f)$, there are at most $2j$ choices for it, or the two positions coincide.
\end{lemma}

\begin{proof}
If $v$ is a tree node with reticulation children $r_1,r_2$, possibly equal, take $f=(v,r_1)$. After the suppression of $v$, the tail position is the edge from the parent of $v$ to $r_2$, an in-edge of a reticulation, or the edge of the parallel pair when $r_1=r_2$. If $v$ is a reticulation with reticulation child $r'$, its parents are tree nodes, since a reticulation parent would be an earlier violation. Take $f$ to be an in-edge of $v$. After the suppression of $v$, the head position is an in-edge of $r'$.
\end{proof}

\begin{lemma}\label{lem:nontc}
There is $c_1>0$ such that, uniformly for $1\le j\le c_1\sqrt\ell$,
\[
|\cP^*_{\ell,j}|\le\Bigl(1+\frac{Cj^2}{\ell}\Bigr)TC_{\ell,j}.
\]
In particular $|\cP_{\ell,j}|=TC_{\ell,j}(1+O(j^2/\ell))$ and $|\cP^*_{\ell,j}|=|\cP_{\ell,j}|(1+O(j^2/\ell))$ in this range.
\end{lemma}

\begin{proof}
Set $A_j=|\cP^*_{\ell,j}|/TC_{\ell,j}$, so $A_0=1$. By Lemma~\ref{lem:pinning}, every structure with a violation is determined by an element of $\cP^*_{\ell,j-1}$, a pinned position and a second position. Hence
\[
A_j\le1+2j\,E_{j-1}\,A_{j-1}\,\frac{TC_{\ell,j-1}}{TC_{\ell,j}}\le1+\frac{C_*j^2}{\ell}A_{j-1},
\]
using the second estimate following \eqref{eq:TCratio}. If $C_*j^2/\ell\le\frac13$ and $A_{j-1}\le\frac32$, then $A_j\le\frac32$. Induction gives $A_j\le\frac32$ throughout the range, and then $A_j\le1+\frac32C_*j^2/\ell$.
\end{proof}

Lemma~\ref{lem:nontc} recovers, with an explicit error, the first-order asymptotic equivalence of $|\cP_{\ell,k}|$ and $TC_{\ell,k}$ for $k=o(\sqrt\ell)$ proved in \cite[Thm.~4.2]{YZasym}, and the fixed-$k$ comparison of \cite[Prop.~4.8]{Mansouri22}. Related first-order results are the asymptotics of normal and hybridization networks for $k=o(\ell^{1/3})$ obtained from decorated trees by Fuchs, Steel and Zhang \cite{FSZ}, and the equivalence with galled tree-child networks by Chang, Fuchs and Yu \cite{CFY}.

\subsection{An insertion identity}

An edge $(p,q)$ of a tree-child network is free if $p$ is a tree node both of whose children are tree nodes or leaves. A network of $TC_{\ell,j-1}$ has $2(\ell-j)$ free edges and $2\ell+j-2$ tree edges \cite[Prop.~5.1]{YZasym}. For a free edge $e=(p,q)$ let $\mathrm{low}(e)$ be the number of tree edges below $q$.

\begin{lemma}\label{lem:identity}
Every $N\in TC_{\ell,j}$ arises in exactly $2j$ ways as $I(N'',e_t,e)$ with $N''\in TC_{\ell,j-1}$, $e$ a free edge, and $e_t\ne e$ a tree edge not below $e$. Consequently \cite[Prop.~5.3]{YZasym}
\begin{equation}\label{eq:lowsum}
\sum_{N''\in TC_{\ell,j-1}}\sum_{e\ \mathrm{free}}\mathrm{low}(e)=2(\ell-j)(2\ell+j-3)\,TC_{\ell,j-1}-2j\,TC_{\ell,j},
\end{equation}
and, uniformly for $1\le j=o(\sqrt\ell)$, this sum is $O\bigl((\ell^{3/2}+j\ell)\,TC_{\ell,j-1}\bigr)$.
\end{lemma}

\begin{proof}
Removing an in-edge $(t,h)$ of any reticulation $h$ of $N$ yields a tree-child network. The other child $b$ of $t$ is not a reticulation, so the parent of $t$ does not become a violation and no parallel edge forms. The child $d$ of $h$ is not a reticulation either, so the other parent of $h$ does not become a violation. Hence each of the $2j$ in-edges gives an insertion representation. Conversely, an insertion into $N''\in TC_{\ell,j-1}$ is tree-child exactly when the head lies on a free edge and the tail lies on a tree edge not below the head and different from it. Counting the insertions in two ways gives $2j\,TC_{\ell,j}=\sum_{N''}\sum_e(2\ell+j-3-\mathrm{low}(e))$, which is \eqref{eq:lowsum}. Inserting \eqref{eq:TCratio} into $2jTC_{\ell,j}=4\ell(\ell-j+1)TC_{\ell,j-1}(1+O(\ell^{-1/2}))$ shows that the leading terms $4\ell^2$ cancel. The remainder is $O(\ell^{3/2}+j\ell)$ times $TC_{\ell,j-1}$.
\end{proof}

For a structure $S$ let $\mathrm{cp}(S)$ be the number of ordered pairs of edges $(e',e)$ such that $e$ lies below $e'$. For $N''\in TC_{\ell,j-1}$, a free edge $e$ has at most $\mathrm{low}(e)+2j+1$ edges below it, since all other edges below it enter reticulations. There are at most $5j$ edges that are not free, and each has at most $E_{j-1}$ edges below it. Hence
\begin{equation}\label{eq:cpbound}
\mathrm{cp}(N'')\le\sum_{e\ \mathrm{free}}\bigl(\mathrm{low}(e)+2j+1\bigr)+5j\,E_{j-1},\qquad
\sum_{N''\in TC_{\ell,j-1}}\mathrm{cp}(N'')=O\bigl((\ell^{3/2}+j\ell)\,TC_{\ell,j-1}\bigr),
\end{equation}
by Lemma~\ref{lem:identity}.

\subsection{The excess of non-tree-child networks}

\begin{theorem}\label{thm:gamma}
Uniformly for $1\le k=o(\sqrt\ell)$,
\begin{equation}\label{eq:gammaall}
|\cP_{\ell,k}|-TC_{\ell,k}=4(k-1)\,\ell\,TC_{\ell,k-1}\Bigl(1+O\bigl(\ell^{-1/2}+\tfrac{k^2}{\ell}\bigr)\Bigr)=\frac{2k(k-1)}{\ell}\,TC_{\ell,k}\Bigl(1+O\bigl(\ell^{-1/2}+\tfrac{k^2}{\ell}\bigr)\Bigr).
\end{equation}
\end{theorem}

\begin{proof}
A network of $\cP_{\ell,k}$ is not tree-child if and only if it has a violation. Split these networks into three sets: $V^{\rm st}$, those with exactly one violation, which is a reticulation (a stack); $V^{\rm tr}$, those with exactly one violation, which is a tree node; and $V^{\ge2}$, those with at least two violations. Throughout, $E_{k-1}=2\ell(1+O(k/\ell))$.

Consider first $V^{\ge2}$. Let $v$ be the first violation and $f$ the edge of Lemma~\ref{lem:pinning}. If $R(N,f)$ still has a violation, $N$ is determined by a structure of $\cP^*_{\ell,k-1}$ with a violation, a pinned position and a second position. By Lemma~\ref{lem:nontc} there are at most $2k\,E_{k-1}\,Ck^2\ell^{-1}TC_{\ell,k-1}=O(k^3TC_{\ell,k-1})$ such $N$. Otherwise the removal destroyed a second violation $v'$. The vertices whose children change under the removal are the suppressed vertices and their neighbours. When $v$ is a stack, $v'$ must be the suppressed tree node $t$, and the free position is then an in-edge of the reticulation that was the other child of $t$. When $v$ is a tree node, $v'$ is the other parent $x$ of the removed reticulation or that reticulation itself. The free position is then the edge from $x$ to the former child, which is an out-edge of a reticulation, an in-edge of a reticulation, or an out-edge of a parent of a reticulation. In every case it is one of $O(k)$ edges, so there are $O(k^2\,TC_{\ell,k-1})$ such $N$. Hence $|V^{\ge2}|=O(k^3\,TC_{\ell,k-1})$.

Consider next $V^{\rm st}$. Let $r\to r'$ be the stack and $x$ the other parent of $r'$. Removing $(x,r')$ yields $N'\in\cP^*_{\ell,k-1}$, the reticulation $r$ of $N'$, and the edge $e$ of $N'$ that carried $x$. The networks for which $N'$ is not tree-child have a destroyed or created violation adjacent to the insertion and are counted as above, $O(k^3\,TC_{\ell,k-1})$ in number. Conversely, for $N'\in TC_{\ell,k-1}$, a reticulation $r$ of $N'$ and an edge $e$, place a new reticulation on the out-edge of $r$ and a new tree node on $e$, joined by an edge. The result lies in $V^{\rm st}$ unless $e$ has its tail at $r$ or below $r$ (cycle or parallel pair), or $e$ enters a reticulation (second violation). Let $d(r)$ be the number of edges with tail at or below $r$. The decomposition is canonical, so
\[
|V^{\rm st}|=\sum_{N'\in TC_{\ell,k-1}}\sum_r\bigl(E_{k-1}-d(r)-O(k)\bigr)+O(k^3\,TC_{\ell,k-1}).
\]
By Lemma~\ref{lem:identity}, the pair $(N',r)$ arises from exactly two insertions into some $N''\in TC_{\ell,k-2}$, with $r$ placed on a free edge $e''$ of $N''$. In such an insertion the new tail cannot lie below $r$, so $d(r)\le\mathrm{low}(e'')+2k+1$. Hence
\begin{align*}
\sum_{N'}\sum_rd(r)&\le\tfrac12E_{k-2}\sum_{N''\in TC_{\ell,k-2}}\sum_{e''\ \mathrm{free}}\bigl(\mathrm{low}(e'')+2k+1\bigr)\\
&=O\bigl(\ell(\ell^{3/2}+k\ell)\,TC_{\ell,k-2}\bigr)=O\bigl(k(\ell^{1/2}+k)\,TC_{\ell,k-1}\bigr),
\end{align*}
by Lemma~\ref{lem:identity} and \eqref{eq:TCratio}. Therefore $|V^{\rm st}|=2(k-1)\ell\,TC_{\ell,k-1}(1+O(\ell^{-1/2}+k^2/\ell))$.

Consider finally $V^{\rm tr}$. Let $u$ be the violation, with reticulation children $r_1,r_2$. Removing $(u,r_1)$ yields $N'$, the reticulation $r_2$ with the in-edge $(y,r_2)$ that carried $u$, and the edge $e$ that carried $r_1$. Removing $(u,r_2)$ yields a second, different decomposition, and the networks for which either removal is not tree-child number $O(k^3TC_{\ell,k-1})$. Conversely, take $N'\in TC_{\ell,k-1}$, one of the $2(k-1)$ in-edges $(y,\rho)$ of its reticulations, and an edge $e$. The insertion of a tree node on $(y,\rho)$ and a reticulation on $e$ produces a network of $V^{\rm tr}$ unless one of three things happens. The edge $e$ may lie above $y$; let $a(y)$ count such edges. Or $e$ enters a reticulation, or the tail of $e$ has a reticulation child; these are $O(k)$ edges. Each network of $V^{\rm tr}$ with two tree-child removals is obtained exactly twice. To bound $\sum a(y)$, represent $(N',\rho)$ by an insertion into $N''\in TC_{\ell,k-2}$ as before. The parents of $\rho$ are the tail of the free edge carrying $\rho$ and the new tail node, so each $a(y)$ is at most the number of edges above an edge of $N''$, plus a constant. Summing over the free edge, the tail edge and $N''$, by \eqref{eq:cpbound},
\[
\sum_{N'}\sum_{(y,\rho)}a(y)\le C\ell\sum_{N''}\mathrm{cp}(N'')+O(\ell^2\,TC_{\ell,k-2})=O\bigl(k(\ell^{1/2}+k)\,TC_{\ell,k-1}\bigr).
\]
Therefore $|V^{\rm tr}|=\frac12\cdot2(k-1)\cdot2\ell\,TC_{\ell,k-1}(1+O(\ell^{-1/2}+k^2/\ell))=2(k-1)\ell\,TC_{\ell,k-1}(1+O(\ell^{-1/2}+k^2/\ell))$.

Adding the three contributions gives the first equality in \eqref{eq:gammaall}. Then \eqref{eq:TCratio} gives $4(k-1)\ell\cdot\frac{k}{2\ell(\ell-k+1)}=\frac{2k(k-1)}{\ell}(1+O(k/\ell))$, which is the second.
\end{proof}

\begin{corollary}\label{cor:threeterm}
Uniformly for $1\le k=o(\sqrt\ell)$,
\begin{multline}\label{eq:threetermU}
|\cP_{\ell,k}|=\binom{\ell}{k}\sqrt2\,e^{-\ell}(2\ell)^{\ell+k-1}\\
\times\Bigl(1-\frac{\sqrt\pi\,k}{2\sqrt\ell}+\frac{62k^2-74k+11}{24\,\ell}+O\Bigl(\frac{k^3}{\ell^{3/2}}+\frac{k^4}{\ell^2}\Bigr)\Bigr).
\end{multline}
For fixed $k$ this reads
\begin{equation}\label{eq:threeterm}
|\cP_{\ell,k}|=\frac{2^{k-1}\sqrt2}{k!}\Bigl(\frac2e\Bigr)^{\ell}\ell^{\ell+2k-1}\Bigl(1-\frac{k\sqrt\pi}{2}\ell^{-1/2}+\frac{50k^2-62k+11}{24}\ell^{-1}+O(\ell^{-3/2})\Bigr).
\end{equation}
\end{corollary}

\begin{proof}
The uniform tree-child expansion of \cite[Prop.~4.1]{YZasym} has the bracket $1-\frac{\sqrt\pi k}{2\sqrt\ell}+\frac{14k^2-26k+11}{24\ell}+O(k^3\ell^{-3/2})$. Multiplying it by $1+2k(k-1)\ell^{-1}+O(k^2\ell^{-3/2}+k^4\ell^{-2})$ from Theorem~\ref{thm:gamma}, and using $\frac{14k^2-26k+11}{24}+2k(k-1)=\frac{62k^2-74k+11}{24}$, gives \eqref{eq:threetermU}. For fixed $k$, $\binom{\ell}{k}=\frac{\ell^k}{k!}(1-\frac{k(k-1)}{2\ell}+O(\ell^{-2}))$ converts the coefficient to $\frac{50k^2-62k+11}{24}$. This agrees with \eqref{eq:TCexp} plus $2k(k-1)$.
\end{proof}

For $k\le4$ the coefficient of $\ell^{-1}$ in \eqref{eq:threeterm} reproduces the exact formulas of \cite{Zhang19,Mansouri22,Yu26,YuZhang26}: the excess of the second coefficient of the $(2\ell-1)!!$ sector over its tree-child value equals $2k(k-1)a_k$, that is, $0,4,8,8$ for $k=1,2,3,4$.

The proof of Theorem~\ref{thm:gamma} also determines the leading asymptotics of reticulation-visible networks. A reticulation $r$ is visible if some leaf $\lambda$ has every root-to-$\lambda$ path passing through $r$, and a network is reticulation-visible if all its reticulations are visible. Write $RV_{\ell,k}$ for the number of such networks and $\Delta_k(\ell)=RV_{\ell,k}-TC_{\ell,k}$.

\begin{corollary}\label{cor:RV}
Uniformly for $2\le k=o(\sqrt\ell)$,
\[
\Delta_k(\ell)=2(k-1)\,\ell\,TC_{\ell,k-1}\bigl(1+O(\ell^{-1/2}+k^2\ell^{-1})\bigr)=\frac{k(k-1)}{\ell}\,TC_{\ell,k}\bigl(1+O(\ell^{-1/2}+k^2\ell^{-1})\bigr).
\]
\end{corollary}

\begin{proof}
Every tree-child network is reticulation-visible, so $\Delta_k(\ell)$ counts the reticulation-visible networks among the sets $V^{\rm st}$, $V^{\rm tr}$ and $V^{\ge2}$ of the proof of Theorem~\ref{thm:gamma}.

Let $N\in V^{\rm tr}$ and let $r$ be a reticulation. Its only violation $u$ is a tree node, so $r$ is not a violation and its child is a tree node or a leaf. Following children that are tree nodes, one reaches a leaf $\lambda$ along a path $r\to c\to\dots\to\lambda$ whose vertices after $r$ have a single parent each. Hence every path from the root to $\lambda$ passes through $r$, so $r$ is visible and $V^{\rm tr}\subset RV$.

Let $N\in V^{\rm st}$, with stack $r\to r'$, and let $x$ be the other parent of $r'$. The vertex $x$ is not below $r$, since the only child of $r$ is $r'$ and $x$ is a parent of $r'$. A root-to-$x$ path therefore avoids $r$, and so does its extension through $r'$ to any leaf below $r$. Hence $r$ is not visible, and $V^{\rm st}\cap RV=\emptyset$.

Therefore $\Delta_k(\ell)=|V^{\rm tr}|+O(|V^{\ge2}|)$, and the two estimates follow from those for $|V^{\rm tr}|$ and $|V^{\ge2}|$ in the proof of Theorem~\ref{thm:gamma}, together with \eqref{eq:TCratio}.
\end{proof}

For $k=2,3$ Corollary~\ref{cor:RV} agrees with the exact formulas for $\Delta_2$ and $\Delta_3$ in \cite{Bat26}. There, $\Delta_k(\ell)=A_k(\ell)(2\ell-3)!!-B_k(\ell)2^{\ell-k+1}\ell!$ with the conjectured leading coefficient $\mathrm{lead}(A_k)=2^k$ for all $k\ge2$, equivalent to $\Delta_k/TC_{\ell,k}\sim k!/\ell$ \cite[Conj.~5.1]{Bat26}. Corollary~\ref{cor:RV} gives instead $\mathrm{lead}(A_k)=2^k/(k-2)!$, which coincides with $2^k$ only for $k=2,3$. For $k=4$ the leading coefficient is $8$, not $16$, and the leading-order law is $k(k-1)/\ell$ for every $k$. Theorem~\ref{thm:gamma} also quantifies the concentration of the uniform distribution on $\cP_{\ell,k}$ on the tree-child networks, in the spirit of the question raised in \cite{CFY} about concentration on smaller classes: the non-tree-child networks form a fraction $2k(k-1)/\ell$ of the tree-child networks to leading order.

\subsection{Symmetric networks}

Let $\sigma\in\Aut(N)$ be an involution and $M$ its set of moved vertices. By Lemma~\ref{lem:seams}, a fixed vertex adjacent to $M$ is either a splitter, namely a tree node whose two children are $x$ and $\sigma(x)$ with $x\in M$, or a seam, namely a reticulation whose two parents are $x$ and $\sigma(x)$. Join $x,y\in M$ when they are adjacent in $N$ or when $y=\sigma(x)$, and let $\beta$ be a connected component of the resulting graph; it is invariant under $\sigma$. A half $H\subset\beta$ contains one vertex from each orbit $\{x,\sigma(x)\}\subset\beta$. It is chosen in a topological order. A tree node is chosen so that its parent lies in $H$ or is a splitter, which is possible because the parent of $\sigma(x)$ is $\sigma$ of the parent of $x$. A reticulation is chosen so that at least one of its parents lies in $H$ or is a splitter, which is possible because if both parents of $x$ lie in $\sigma(H)$, then both parents of $\sigma(x)$ lie in $H$.

Let $z$ and $s$ be the numbers of splitters and seams adjacent to $\beta$. Let $t_H$ and $m_r$ be the numbers of tree nodes and reticulations in $H$, and let $c_r$ be the number of reticulations of $H$ with a parent in $\sigma(H)$. Write $\lambda=s+m_r+c_r$.

\begin{lemma}\label{lem:blockid}
$s=z+t_H-m_r$; hence $\lambda=z+t_H+c_r$, and the number of reticulations involved, $s+2m_r$, equals $z+|H|$. Moreover $z\ge1$, and $z\ge2$ when $t_H=0$.
\end{lemma}

\begin{proof}
Each splitter adjacent to $\beta$ sends exactly one edge into $H$, and each seam receives exactly one edge from $H$. No other edge joins $H$ to a fixed vertex, by Lemma~\ref{lem:seams}. The involution maps the edges from $H$ to $\sigma(H)$ bijectively onto the edges from $\sigma(H)$ to $H$; let $c$ be their common number and $e_{\rm int}$ the number of edges inside $H$. Counting in-degrees of $H$ gives $t_H+2m_r=z+e_{\rm int}+c$. Counting out-degrees gives $2t_H+m_r=e_{\rm int}+s+c$. Subtracting yields $s=z+t_H-m_r$.

A path from the root to a vertex of $\beta$ enters the moved set through a splitter of $\beta$, so $z\ge1$. The first moved vertex of $\beta$ in a topological order has only fixed parents, which are splitters. If $t_H=0$, all moved vertices are reticulations, so this vertex has two distinct splitter parents.
\end{proof}

\begin{lemma}\label{lem:blockcount}
Fix the isomorphism type of the configuration $B=\beta\cup\{\text{splitters}\}\cup\{\text{seams}\}$. Let $F$ be the number of splitters whose child in $H$ is a tree node or a reticulation counted in $c_r$, and let $p$ be the number of the remaining splitters. The number of $N\in\cP_{\ell,k}$ that admit an involution with a component of this type is at most
\[
2^{4|B|}\,(2k)^{p}\,(E_{k-\lambda})^{F}\,|\cP^*_{\ell,k-\lambda}|.
\]
Moreover $F-2\lambda\le-3$, with equality only when $\beta$ is the pair of twins of a crossed cherry $\Gamma_1$, and $F\le z\le\lambda$.
\end{lemma}

\begin{proof}
Delete $\sigma(H)$ and suppress every vertex left with in-degree and out-degree one. These are the splitters, the seams, and the reticulations of $H$ with a parent in $\sigma(H)$. Every other reticulation of $H$ keeps two parents, and every tree node of $H$ keeps its parent, by the choice of $H$. No vertex outside $B$ changes its degrees, so the result $R$ lies in $\cP^*_{\ell,k-\lambda}$.

The network $N$ is recovered from $R$, the type of $B$, and the positions of the surviving vertices of $H$ in $R$, in the following way. Decompose $H$ into pieces, the connected components of the subgraph of $N$ spanned by $H$. The first vertex of a piece in a topological order has no parent in the piece. By the choice of $H$, one of its parents is a splitter; distinct pieces have distinct splitters, since a splitter has one child in $H$. If that first vertex is a reticulation surviving in $R$, it is one of at most $k$ reticulations of $R$; otherwise its position is one of the $E_{k-\lambda}$ edges of $R$. This gives at most $(2k)^{p}E_{k-\lambda}^{F}$ choices for the first vertices of all pieces.

Every other vertex of a piece is reached from its first vertex by a directed path of length at most $|B|$ inside the piece. In $R$ this path survives, shortened by suppressions, and since out-degrees are at most two there are at most $2^{|B|}$ choices for each vertex. The suppressed vertices of $B$ sit on edges of $R$ incident with located vertices: splitters on in-edges, seams on out-edges, and suppressed reticulations of $H$ between located vertices. Their positions and their order along an edge are fixed by at most $2^{|B|}$ choices each. Finally $\sigma(H)$ is a copy of $H$, and its edges to $H$, to the splitters and to the seams are given by the type of $B$. The product of these choices is at most $2^{4|B|}(2k)^pE_{k-\lambda}^F$.

For the inequality, $F\le z$, and Lemma~\ref{lem:blockid} gives
\[
F-2\lambda\le z-2(z+t_H+c_r)=-z-2t_H-2c_r .
\]
If $t_H\ge2$ this is at most $-5$. Suppose $t_H=1$. If $z\ge2$ or $c_r\ge1$ the bound is at most $-4$. If $z=1$ and $c_r=0$, the parents of reticulations of $H$ lie in $H$ or are splitters. The child in $H$ of the single splitter is then the tree node of $H$, because a reticulation child would need a second parent above it in $H$, and there is none. A first reticulation of $H$ would have both parents equal to that tree node, which is a parallel pair, so $m_r=0$. Then $s=2$, and $\beta$ is the twin pair of $\Gamma_1$, with $F-2\lambda=1-4=-3$. Suppose $t_H=0$; then $z\ge2$. If $c_r\ge1$ the bound is at most $-4$. If $c_r=0$, every child in $H$ of a splitter is a reticulation surviving in $R$, so $F=0$ and $F-2\lambda=-2z\le-4$.
\end{proof}

\begin{theorem}\label{thm:symlaw}
For every fixed $k\ge2$, the probability that a uniformly random network of $\cP_{\ell,k}$ has a nontrivial automorphism group, and its defect density, are
\begin{equation}\label{eq:symlaw}
\frac{k(k-1)}{4\ell^3}\bigl(1+O(\ell^{-1/2})\bigr)\qquad\text{and}\qquad\frac{D_k(\ell)}{|\cP_{\ell,k}|}=\frac{k(k-1)}{8\ell^3}\bigl(1+O(\ell^{-1/2})\bigr).
\end{equation}
Both statements hold uniformly for $2\le k\le\ell^{1/100}$, with relative error $O(\ell^{-1/2}+k^6\ell^{-1})$.
\end{theorem}

\begin{proof}
By Lemma~\ref{lem:autbound}, a nontrivial group contains an involution, and every one of its components falls under Lemma~\ref{lem:blockcount}. For $k\le c_1\sqrt\ell$, Lemma~\ref{lem:nontc} and the second estimate following \eqref{eq:TCratio} give $|\cP^*_{\ell,k-\lambda}|\le(Ck/\ell^2)^{\lambda}|\cP_{\ell,k}|$.

Consider first the configurations other than $\Gamma_1$ with $\lambda\le4$. Their types form a finite set, independent of $k$, and $p\le\lambda$. By Lemma~\ref{lem:blockcount} they account for at most $Ck^{2\lambda}\ell^{F-2\lambda}|\cP_{\ell,k}|\le Ck^8\ell^{-4}|\cP_{\ell,k}|$ networks.

Consider next the configurations with $\lambda\ge5$. Here $|B|\le2|H|+z+s\le8\lambda$ and $\lambda\le k$. The number of types with $|B|\le8\lambda$ is at most $(Ck)^{24\lambda}$, since every vertex has at most three neighbours inside $B$. Using $F\le\lambda$, these configurations account for at most
\[
\sum_{\lambda\ge5}(Ck)^{24\lambda}C^{\lambda}k^{\lambda}\ell^{\lambda}\Bigl(\frac{Ck}{\ell^2}\Bigr)^{\lambda}|\cP_{\ell,k}|\le\sum_{\lambda\ge5}\Bigl(\frac{(Ck)^{26}}{\ell}\Bigr)^{\lambda}|\cP_{\ell,k}|
\]
networks. For $k\le\ell^{1/100}$ and $\ell$ large this is at most $2(Ck)^{130}\ell^{-5}|\cP_{\ell,k}|$. Relative to $k^2\ell^{-3}|\cP_{\ell,k}|$ this is $O(k^{128}\ell^{-2})=O(\ell^{-1/2})$. Networks whose involutions have two or more components are bounded in the same way, by applying the deletion to each component.

The networks with a $\Gamma_1$ component are obtained by expanding a tree node $u$ of some $N'\in\cP^*_{\ell,k-2}$, in at most $(\ell+k-3)|\cP^*_{\ell,k-2}|$ ways. Conversely, every pair of an asymmetric $N'\in\cP_{\ell,k-2}$ and one of its $\ell+k-3$ tree nodes gives such a network. Distinct pairs give distinct networks, except for networks with two components, which were bounded above. Hence their number is $\ell\,|\cP_{\ell,k-2}|(1+O(k^2/\ell))$, and their group has order two except on a set already bounded. Finally, by Lemma~\ref{lem:nontc} and \eqref{eq:TCratio},
\[
\frac{\ell\,|\cP_{\ell,k-2}|}{|\cP_{\ell,k}|}=\frac{k(k-1)}{4\ell^3}\Bigl(1+O\bigl(\ell^{-1/2}+\tfrac{k^2}{\ell}\bigr)\Bigr).
\]
The non-$\Gamma_1$ contributions are smaller by the relative factor $O(k^6/\ell)$. Each network with a group of order two contributes $\frac12$ to $D_k$.
\end{proof}

For $k=2,3$, Theorem~\ref{thm:symlaw} agrees with the exact defects of Theorem~\ref{thm:defects}. The exponent $1/100$ results from the crude count of configuration types and has not been optimised.

\section{Networks with a distinguished involution}\label{sec:symrec}

By Lemma~\ref{lem:exponent2}, for at most five reticulations every nontrivial automorphism is an involution. This reduces the defect to a weighted count of networks with a distinguished involution.

\begin{corollary}\label{cor:Dinv}
For $k\le5$,
\begin{equation}\label{eq:Dinv}
D_k(\ell)=\sum_{N\in\cP_{\ell,k}}\frac{\#\{\sigma\in\Aut(N):\sigma^2=\mathrm{id}\ne\sigma\}}{|\Aut(N)|}.
\end{equation}
By the orbit--stabiliser argument of Lemma~\ref{lem:vertexlabelled}, $D_k(\ell)$ equals $\ell!/n!$ times the number of vertex-labelled networks with $n=2\ell+2k-1$ vertices equipped with a nontrivial involution of the internal labels that is an automorphism.
\end{corollary}

\begin{proof}
In $\mathbb{Z}_2^m$ every nonidentity element is an involution, so $1-|\Aut(N)|^{-1}$ equals the number of involutions divided by $|\Aut(N)|$.
\end{proof}

For $k\ge6$ the identity \eqref{eq:Dinv} fails in general: the group of Example~\ref{ex:dihedral} has order eight but contains only five involutions.

The right-hand side of \eqref{eq:Dinv} counts networks with a distinguished involution. Such pairs admit a layer-peeling recursion of the type of Theorem~\ref{thm:recursion}, and the present section derives it.

\begin{definition}\label{def:invstruct}
Let $a_f,a_s\ge0$ with $a=a_f+2a_s\ge1$, and let $\rho$ be the involution of $[a]$ that fixes $1,\dots,a_f$ and exchanges $a_f+2i-1$ with $a_f+2i$ for $1\le i\le a_s$. An involution of an $a$-rooted structure $S$ is an incidence-preserving involution of its vertices and edges that fixes every leaf and acts on the roots as $\rho$. For such a pair $(S,\sigma)$, let $P_f$ be the number of parallel pairs whose two edges are fixed by $\sigma$, and $P_x$ the number of parallel pairs with fixed endpoints whose two edges are exchanged. Let $P_s$ be the number of two-element orbits of parallel pairs, and $w$ the number of two-element orbits of $\sigma$ on edges. Define
\[
Q^\sigma_{a_f,a_s}(\ell,m;P_f,P_x,P_s;w)=\sum_{S}\frac{\#\{\sigma\}}{|\Aut_e(S)|},
\]
where $S$ runs over isomorphism classes of $a$-rooted structures with $\ell$ leaves and $m$ reticulations, and $\#\{\sigma\}$ is the number of involutions of $S$ with the given parameters. Set $Q^\sigma_{0,0}(\ell,m;P;w)=[\ell=m=0,\,P=0,\,w=0]$.
\end{definition}

For $a_f=1$, $a_s=0$ and $P=0$ the structures are networks, and an involution is an element $\sigma\in\Aut(N)$ with $\sigma^2=\mathrm{id}$. Moreover $w=0$ exactly when $\sigma=\mathrm{id}$. Hence
\begin{equation}\label{eq:QsigmaTop}
Q^\sigma_{1,0}(\ell,k;0,0,0;0)=W_k(\ell),\qquad
\sum_{w\ge1}Q^\sigma_{1,0}(\ell,k;0,0,0;w)=\sum_{N\in\cP_{\ell,k}}\frac{\#\{\sigma:\sigma^2=\mathrm{id}\ne\sigma\}}{|\Aut(N)|},
\end{equation}
and for $k\le5$ the second sum is $D_k(\ell)$ by Corollary~\ref{cor:Dinv}.

An involution $\sigma$ of $S$ maps the layer $\Lambda(S)$ onto itself, because it maps tree components to tree components and roots to roots. It therefore permutes the closed reticulations and the open reticulations, and it restricts to an involution $\sigma'$ of the remainder that permutes its roots as it permutes the closed reticulations. The closed reticulations fall into three types. There are $c_f$ fixed ones whose two slots are fixed, $c_x$ fixed ones whose two slots are exchanged (seams of the layer), and $c_s$ two-element orbits. The open reticulations comprise $b_f$ fixed ones and $b_s$ two-element orbits. We write $a'=c_f+c_x+2c_s$ and $b'=b_f+2b_s$. The slots of the layer forest are then $2c_f+b_f$ fixed slots and $q=c_x+2c_s+b_s$ pairs of exchanged slots.

Write $t_m=(2m-3)!!$ for $m\ge1$, and
\[
G_R(u)=\sum_{m\ge1}\frac{2^m t_m}{m!}u^m,\qquad G_F(u)=\sum_{m\ge1}\frac{2^{m-1}t_m}{m!}u^m .
\]

\begin{lemma}\label{lem:invforest}
Consider forests with $a$ labelled roots, permuted by $\rho$, and with $n_0$ labelled fixed leaves and $q$ labelled pairs of leaves exchanged by a prescribed involution. The number of such forests invariant under an involution that extends these data, and which have exactly $\varphi$ fixed vertices whose two children are exchanged (forks), is
\begin{equation}\label{eq:invforest}
\Lambda_\varphi=a_f!\,F_{n_0+\varphi,\,a_f}\cdot\frac{q!}{\varphi!}\,[u^q]\,G_R(u)^{a_s}\,G_F(u)^{\varphi}.
\end{equation}
Each such forest has $2q-a_s-\varphi$ two-element orbits of edges.
\end{lemma}

\begin{proof}
A leaf-labelled forest with labelled roots has no nontrivial automorphism fixing its leaves and roots. An extending involution is therefore unique if it exists, and the lemma counts invariant forests.

A vertex whose image is a different vertex has no fixed leaf below it in its own subtree, since fixed leaves have fixed ancestors along their unique root path. The trees at two exchanged roots are therefore images of each other and carry only exchanged leaves. Such a pair is determined by one tree on a set $A$ that contains one leaf from each of $m\ge1$ pairs; there are $2^m$ choices of $A$ and $t_m$ trees on $A$.

In a tree with a fixed root, the fixed vertices form a subtree containing the root. A fixed vertex has either two fixed children or two exchanged children, because its children are permuted among themselves. In the second case it is a fork, and its two subtrees are images of each other, carrying exchanged leaves only. The fork is determined by the unordered pair $\{A,\sigma(A)\}$ and a tree on $A$, giving $2^{m-1}t_m$ possibilities for $m$ pairs.

Contracting each fork with its two subtrees to a single leaf leaves a forest with $a_f$ labelled roots on the $n_0$ fixed leaves and the $\varphi$ fork leaves. There are $a_f!\,F_{n_0+\varphi,a_f}$ such forests. The $q$ pairs are distributed into $a_s$ labelled nonempty blocks for the exchanged root pairs and $\varphi$ unordered nonempty blocks for the forks, which gives the exponential-formula coefficient in \eqref{eq:invforest}.

An edge is moved exactly when its head is moved. A block of $m$ pairs consists of two exchanged trees with $2m-1$ non-root vertices each, giving $2m-1$ two-element orbits of edges. Summing over the $a_s+\varphi$ blocks gives $2q-a_s-\varphi$.
\end{proof}

A parallel pair of the layer arises as a cherry formed by the two slots of one closed reticulation. For a reticulation of type $c_f$ the cherry lies among the fixed vertices and has both edges fixed. For type $c_x$ it is a fork with one pair, and its edges are exchanged. For type $c_s$ it lies inside a block, with its image cherry in the mirror subtree. Prescribing $s_f$, $s_x$ and $s_s$ such cherries amounts to the following replacements:
\begin{itemize}
\item contracting $s_f$ fixed cherries, so that $n_0$ decreases by $s_f$;
\item reserving $s_x$ forks with one pair each, so that $q$ decreases by $s_x$ and $s_x$ fork leaves are added;
\item merging $s_s$ couples of pairs into single pairs, so that $q$ decreases by $s_s$.
\end{itemize}
The number of two-element edge orbits remains $2q-a_s-\varphi_{\rm tot}$ with the original $q$, where $\varphi_{\rm tot}$ counts all forks. Inclusion–exclusion over the prescribed sets then gives the number $\Lambda^{\rm ex}(p_f,p_x,p_s;w_L)$ of invariant layer forests with exactly $p_f$, $p_x$ and $p_s$ cherries of the three types and $w_L$ two-element edge orbits.

\begin{lemma}\label{lem:eqplace}
Let $(R,\sigma')$ have $E$ edges, $w$ two-element edge orbits and $P_f,P_x,P_s$ parallel pairs of the three types. Place $b_f$ labelled points on fixed edges and $b_s$ labelled pairs of points on two-element orbits, one point of each pair on each edge of the orbit, each point subdividing its edge. The number of such placements in which prescribed $q_f$, $q_x$ and $q_s$ parallel pairs of the respective types receive no point, and all others receive at least one, is $\Pi_f\Pi_s$, where
\begin{align*}
\Pi_f&=\binom{P_f}{q_f}\sum_{i}(-1)^i\binom{P_f-q_f}{i}\rise{(E-2w-2q_f-2i)}{b_f},\\
\Pi_s&=2^{b_s}\binom{P_x}{q_x}\binom{P_s}{q_s}\sum_{i_1,i_2}(-1)^{i_1+i_2}\binom{P_x-q_x}{i_1}\binom{P_s-q_s}{i_2}\rise{(w-q_x-2q_s-i_1-2i_2)}{b_s}.
\end{align*}
\end{lemma}

\begin{proof}
A fixed point lies on a fixed edge and keeps it fixed after subdivision. A pair of exchanged points lies on a two-element orbit, in $2$ ways, and subdividing the two edges increases the number of such orbits by one. A pair of points on a single fixed edge is impossible: their order along the edge would be reversed by the involution. Hence fixed points see $E-2w$ edges and pairs of points see $w$ orbits, and the sequential counts are the rising factorials. A parallel pair of type $f$ consists of two fixed edges and is broken only by fixed points. A pair of type $x$ is one two-element orbit, broken only by a pair of points. An orbit of type $s$ consists of two two-element orbits of edges, and is broken by a pair of points on either of them. Inclusion–exclusion over the unbroken ones gives the formulas.
\end{proof}

\begin{theorem}\label{thm:symrec}
For $a\ge1$,
\begin{multline}\label{eq:symrec}
Q^\sigma_{a_f,a_s}(\ell,m;P;w)=\sum_{j=0}^{\ell}\binom{\ell}{j}\sum_{c_f,c_x,c_s,b_f,b_s}\frac{1}{c_f!\,c_x!\,c_s!\,b_f!\,b_s!\,2^{b_s}\,2^{a'}}\\
\times\sum\Lambda^{\rm ex}(p;w_L)\;Q^\sigma_{c_f+c_x,\,c_s}(\ell-j,m';P^R;w_R)\;\Pi_f\,\Pi_s .
\end{multline}
Here $m'=m-a'-b'\ge0$, and $\Lambda^{\rm ex}$ is taken with $n_0=j+2c_f+b_f$ and $q=c_x+2c_s+b_s$. The inner sum runs over $p+q=P$ componentwise and $w_L+w_R+b_s=w$. Finally, $\Pi_f\Pi_s$ is evaluated with $E=2(\ell-j)+3m'-a'$, the parallel pairs $P^R$ of the remainder, $b_f$ fixed points and $b_s$ pairs of points.
\end{theorem}

\begin{proof}
The proof has three steps: an expression of $Q^\sigma$ through rigid decorations, the equivariance of the layer decomposition, and the count of auxiliary labellings.

A rigid decoration of a structure $S$ is a bijective labelling of its $N_S$ internal vertices together with an order of the two edges of every parallel pair; $N_S$ and the number $p=P_f+P_x+2P_s$ of parallel pairs are fixed by the state. Fix $S$ and an involution $\sigma$ of $S$. For every rigid decoration $\delta$ there is exactly one permutation of the decoration data that realises $\sigma$ through $\delta$. The group $\Aut_e(S)$ acts freely on the rigid decorations, so their number is $N_S!\,2^p/|\Aut_e(S)|$. Hence the number of decorated pairs, consisting of a rigid decoration and a permutation realising an involution with the prescribed parameters, is $N_S!\,2^p\,\#\{\sigma\}/|\Aut_e(S)|$. Summing over $S$ and dividing by $N_S!\,2^p$ gives $Q^\sigma$.

The decomposition of Theorem~\ref{thm:recursion} applies to a decorated pair. It produces a decorated layer forest, a decorated remainder, a placement, and an auxiliary labelling. The auxiliary labelling consists of labellings of the closed and open reticulations by $[a']$ and $[b']$ and an order of the two slots of every closed reticulation. As in that proof, the correspondence is bijective and the group $S_{a'}\times S_{b'}\times\mathbb{Z}_2^{a'}$ acts freely on the auxiliary data. The involution induces an involution $\pi_c$ of $[a']$ with $c_f+c_x$ fixed points and $c_s$ two-cycles, an involution $\pi_o$ of $[b']$ with $b_f$ fixed points and $b_s$ two-cycles, and a correspondence between slots.

Fix the type $(c_f,c_x,c_s,b_f,b_s)$. The data $\pi_c$, together with the choice of the $c_x$ fixed points whose slots are exchanged, can be chosen in $a'!/(c_f!\,c_x!\,c_s!\,2^{c_s})$ ways. The involution $\pi_o$ can be chosen in $b'!/(b_f!\,b_s!\,2^{b_s})$ ways. For each two-cycle $\{r,r'\}$ of $\pi_c$ the induced bijection between the slots of $r$ and those of $r'$ is one of two. Reversing the slot order of $r'$ interchanges the two choices, so both give the same number of invariant layer forests, and together they contribute a factor $2^{c_s}$. For a fixed choice of these data, the layer is counted by Lemma~\ref{lem:invforest} with the cherry inclusion--exclusion. The remainder is counted by $Q^\sigma_{c_f+c_x,c_s}$, since by relabelling its count depends only on the type of $\pi_c$. The placement is counted by Lemma~\ref{lem:eqplace}. Dividing by the order $a'!\,b'!\,2^{a'}$ of the group acting on the auxiliary data gives
\[
\frac{a'!}{c_f!\,c_x!\,c_s!\,2^{c_s}}\cdot\frac{b'!}{b_f!\,b_s!\,2^{b_s}}\cdot\frac{2^{c_s}}{a'!\,b'!\,2^{a'}}=\frac{1}{c_f!\,c_x!\,c_s!\,b_f!\,b_s!\,2^{b_s}\,2^{a'}},
\]
which is the prefactor in \eqref{eq:symrec}. When the two slots of a closed reticulation form a cherry, the order of the slots coincides with the order of the two edges of the resulting parallel pair, exactly as in the proof of Theorem~\ref{thm:recursion}. The binomial $\binom{\ell}{j}$ distributes the leaf labels between the layer and the remainder.

The two-element edge orbits of $S$ are those of the layer, those of the remainder, and the $b_s$ orbits created by the pairs of points, so $w=w_L+w_R+b_s$. The parallel pairs of $S$ are the cherries of the layer and the unbroken parallel pairs of the remainder, so $P=p+q$.
\end{proof}

The generating-function form of Theorem~\ref{thm:closure} carries over to \eqref{eq:symrec}. The layer factors are derivatives of $T^{a_f}/a_f!$. The factor $\Pi_f$ is a polynomial in the number of leaves of the remainder, and $\Pi_s$ does not depend on it. Hence the exponential generating functions of $Q^\sigma_{a_f,a_s}(\ell,m;P;w)$ in $\ell$ are Laurent polynomials in $u$, and Corollary~\ref{cor:shape} applies to the defect.

Two hand evaluations illustrate the recursion. For $k=2$, $\ell=2$, the only contribution with $w\ge1$ comes from $c_x=2$, a single fork with two pairs, and a remainder of two fixed roots. The weight is $\frac1{2!\,2^2}\cdot2\cdot2=\frac12$, which is $D_2(2)$. For $k=3$, $\ell=1$, three types of involution contribute.
\begin{itemize}
\item The exchange of the twin tree nodes ($c_x=2$) contributes $\frac1{2!\,2^2}\cdot2\cdot1=\frac14$.
\item The exchange of the twin reticulations, alone or combined with that of the tree nodes, contributes through $c_s=1$. The layer has two invariant forests, two singleton forks or one fork with two pairs. The remainder is $Q^\sigma_{0,1}(1,1;0;1)=\frac12\cdot2\cdot1=1$. The contribution is $\frac1{2^2}\cdot2\cdot1=\frac12$.
\end{itemize}
The sum is $\frac34=D_3(1)$, the three nontrivial involutions of $\Gamma_1$ with merged exits each contributing $\frac14$.

\section{Exact counts with four and five reticulations}\label{sec:k45}

The results of this section are exact identities whose proof consists of an exact computation. The recursion of Theorem~\ref{thm:recursion} and the recursion \eqref{eq:symrec} are evaluated with the substitutions of Theorem~\ref{thm:closure}, in exact rational arithmetic in $\mathbb{Q}[u,u^{-1}]$. Every state then carries a Laurent polynomial in place of a sequence in $\ell$. By Corollary~\ref{cor:Dinv}, for $k\le5$ the sectors $w\ge1$ of $Q^\sigma_{1,0}(\cdot,k;0,0,0;w)$ add up to the generating function of $D_k$, and the sector $w=0$ is that of $W_k$. Corollary~\ref{cor:shape} converts each generating function into a closed form. For $k\le3$ this reproduces Theorem~\ref{thm:defects}. For each new result, what was computed and how it was checked are stated with the result.

The asymptotic results of Section~\ref{sec:asympt} provide one of these checks, through the most singular coefficients of the generating functions. For $k\le5$ let $G_k=\sum_ng_nu^n$ be the generating function of $|\cP_{\ell,k}|$ and $G_k^D=\sum_ng^D_nu^n$ that of $D_k$, and put $s=2k-1$. By \eqref{eq:uextract}, the term $g_{-(2s+1)}u^{-(2s+1)}$ contributes $g_{-(2s+1)}(2\ell+2s-1)!!/(2s-1)!!$. The term $g_{-2s}u^{-2s}$ contributes $g_{-2s}2^\ell\ell!\binom{\ell+s-1}{s-1}$, which is smaller by a factor of order $\ell^{-1/2}$. The term $g_{-(2s-1)}u^{-(2s-1)}$ is smaller by a factor of order $\ell^{-1}$. One has $(2\ell+2s-1)!!=(2\ell-1)!!(2\ell)^s\bigl(1+\frac{s^2}{2\ell}+O(\ell^{-2})\bigr)$, $(2\ell-1)!!=\sqrt2(2\ell/e)^\ell\bigl(1-\frac1{24\ell}+O(\ell^{-2})\bigr)$ and $2^\ell\ell!/(2\ell-1)!!=\sqrt{\pi\ell}\,(1+O(\ell^{-1}))$. With these expansions, the three-term expansion \eqref{eq:threeterm} and the defect law of Theorem~\ref{thm:symlaw} are equivalent, at the leading orders, to the four identities
\begin{equation}\label{eq:gtests}
\begin{gathered}
g_{-(2s+1)}=\frac{(4k-3)!!}{2^kk!},\qquad
\frac{g_{-2s}}{g_{-(2s+1)}}\cdot\frac{(2s-1)!!}{2^s(s-1)!}=-\frac k2,\\
\frac{s^2}2-\frac1{24}+\frac{(2s-1)\,g_{-(2s-1)}}{2\,g_{-(2s+1)}}=\frac{50k^2-62k+11}{24},\\
g^{D}_{-(2s-5)}\,(4k-3)(4k-5)(4k-7)=k(k-1)\,g_{-(2s+1)}.
\end{gathered}
\end{equation}
Since the coefficients $g_n$ are produced by the computation and not by the asymptotic analysis, these identities test the computation and the asymptotic theorems against each other.

\begin{corollary}[computer-assisted]\label{cor:D4}
For $\ell\ge2$,
\begin{align*}
D_4(\ell)={}&\frac{(2\ell-5)!!}{24}\bigl(48\ell^6+224\ell^5+348\ell^4-698\ell^3-699\ell^2-483\ell+666\bigr)\\
&-\frac{2^\ell\ell!}{4}\bigl(2\ell^3+11\ell^2+22\ell+13\bigr),
\end{align*}
and $D_4(1)=\frac{13}{4}$.
\end{corollary}

\begin{proof}
By Lemma~\ref{lem:exponent2} and Corollary~\ref{cor:Dinv}, the generating function of $D_4$ is the sum over $w\ge1$ of the generating functions of $Q^\sigma_{1,0}(\ell,4;0,0,0;w)$. The evaluation of \eqref{eq:symrec} with four reticulations gives a Laurent polynomial whose largest odd power is $u^3$ and whose only nonnegative even power contributes at $\ell=1$. Corollary~\ref{cor:shape} converts it into the stated form. For instance, at $\ell=2$ one has $(2\ell-5)!!=1$; the first term equals $7128/24=297$ and the second equals $\frac{117}{4}\cdot8=234$, so $D_4(2)=63$.
\end{proof}

Corollary~\ref{cor:D4} was checked independently in three ways. Added to $W_4$, it reproduces the totals of \cite{YuZhang26} for $\ell\le12$ and the exhaustive enumeration by canonical forms for $\ell\le3$; Table~\ref{tab:k4} shows the first values. The four-reticulation formula of \cite{YuZhang26} is thereby confirmed by a method that shares no step with the component-graph derivation. The defect coincides for $1\le\ell\le12$ with the one obtained from the hand classification of the symmetric networks in Appendix~\ref{app:k4}, which in addition gives the numbers of networks whose automorphism group has order two or four. Finally, the generating function satisfies \eqref{eq:gtests} with $k=4$.

\begin{table}[t]
\centering
\begin{tabular}{crrrc}
\toprule
$\ell$ & $W_4(\ell)$ by \eqref{eq:recursion} & $D_4(\ell)$ & $|\cP_{\ell,4}|$ \cite{YuZhang26} & enumeration\\
\midrule
1 & $423/4$ & $13/4$ & $109$ & $109$\\
2 & $3818$ & $63$ & $3881$ & $3881$\\
3 & $224775/2$ & $2073/2$ & $113424$ & $113424$\\
4 & $12775305/4$ & $73455/4$ & $3212190$ & --\\
5 & $368530695/4$ & $1414245/4$ & $92486235$ & --\\
\bottomrule
\end{tabular}
\caption{Four reticulations: the weighted count, the defect of Corollary~\ref{cor:D4}, the published totals and the exhaustive enumeration.}
\label{tab:k4}
\end{table}

For five reticulations the same computation, carried one reticulation further, gives Table~\ref{tab:k5} and the following closed forms.

\begin{table}[ht]
\centering
\footnotesize
\begin{tabular}{rrrr}
\toprule
$\ell$ & $W_5(\ell)$ & $D_5(\ell)$ & $|\cP_{\ell,5}|$\\
\midrule
1 & $14849/8$ & $319/8$ & 1896\\
2 & $85448$ & $1005$ & 86453\\
3 & $25110903/8$ & $189489/8$ & 3162549\\
4 & $436917765/4$ & $2256975/4$ & 109793685\\
5 & $30364216785/8$ & $112321935/8$ & 3809567340\\
6 & $134631424935$ & $368988165$ & 135000413100\\
7 & $39410646381855/8$ & $82233543105/8$ & 4936609990620\\
8 & $747705526842015/4$ & $1214707427325/4$ & 187230058567335\\
9 & $58988867326614225/8$ & $76036711049775/8$ & 7383113004708000\\
10 & $605380988854528425/2$ & $629250654168675/2$ & 303005119754348550\\
11 & $103464763933400844975/8$ & $87948173668734225/8$ & 12944089013383697400\\
12 & $2299897473814449934425/4$ & $1618553143499967675/4$ & 575379006739487475525\\
\bottomrule
\end{tabular}
\caption{Weighted count, defect and number of networks with five reticulations.}
\label{tab:k5}
\end{table}

\begin{theorem}[computer-assisted]\label{thm:k5}
For $\ell\ge3$,
\[
|\cP_{\ell,5}|=\frac{\mathcal A_5(\ell)}{1890}\,(2\ell-5)!!-\frac{\mathcal B_5(\ell)}{4608}\,2^\ell\ell!,
\]
with
\begin{align*}
\mathcal A_5(\ell)={}&1008\ell^{11}+37968\ell^{10}+597880\ell^9+5064160\ell^8+24476719\ell^7+63651595\ell^6\\
&+57274651\ell^5-107627135\ell^4-290117607\ell^3-158811138\ell^2+68349609\ell+45603810,\\
\mathcal B_5(\ell)={}&1536\ell^8+48960\ell^7+671891\ell^6+5143935\ell^5+23826389\ell^4+67730685\ell^3\\
&+114086072\ell^2+102589956\ell+36927648 .
\end{align*}
Moreover $|\cP_{1,5}|=1896$ and $|\cP_{2,5}|=86453$. The decomposition $|\cP_{\ell,5}|=W_5(\ell)+D_5(\ell)$ is given by
\begin{align*}
W_5(\ell)={}&\frac{(2\ell+3)!!}{7560}\bigl(252\ell^7+9492\ell^6+150100\ell^5+1289140\ell^4+6485153\ell^3\\
&+19078277\ell^2+30330879\ell+20056050\bigr)\\
&-\frac{2^\ell\ell!}{4608}\bigl(1536\ell^8+48960\ell^7+671891\ell^6+5141631\ell^5+23794949\ell^4\\
&+67558221\ell^3+113632232\ell^2+102025716\ell+36673920\bigr)
\end{align*}
for $\ell\ge1$, and
\begin{align*}
D_5(\ell)={}&\frac{(2\ell-5)!!}{24}\bigl(32\ell^8+464\ell^7+2648\ell^6+5904\ell^5+392\ell^4-17783\ell^3\\
&-14943\ell^2+1335\ell+6066\bigr)-\frac{2^\ell\ell!}{96}\bigl(48\ell^5+655\ell^4+3593\ell^3\\
&+9455\ell^2+11755\ell+5286\bigr)
\end{align*}
for $\ell\ge3$, with $D_5(1)=\frac{319}{8}$ and $D_5(2)=1005$.
\end{theorem}

\begin{proof}
The recursions of Theorems~\ref{thm:recursion} and~\ref{thm:symrec} were evaluated with five reticulations as described at the beginning of this section. By Lemma~\ref{lem:exponent2} and Corollary~\ref{cor:Dinv}, the sectors $w\ge1$ give the generating function of $D_5$ and the sector $w=0$ that of $W_5$. Corollary~\ref{cor:shape} converts the generating functions into the stated forms. The largest odd power is $u^{3}$ for $|\cP_{\ell,5}|$ and $D_5$, and $u^{-5}$ for $W_5$; the even powers $u^{2}$ and $u^{4}$ contribute only for $\ell\le2$.
\end{proof}

The leading behaviour of Theorem~\ref{thm:k5} can be followed by hand from the computed coefficients. The three most singular coefficients of $G_5$ are
\[
g_{-19}=\frac{2297295}{256},\qquad g_{-18}=-13440,\qquad g_{-17}=-\frac{225225}{256},
\]
and the most singular coefficient of $G_5^D$ is $g^D_{-13}=\frac{3465}{64}$. By \eqref{eq:uextract}, the term $g_{-19}u^{-19}$ contributes $g_{-19}(2\ell+17)!!/17!!$. Since $(2\ell+17)!!/(2\ell-5)!!$ is a product of eleven factors $2\ell+2i-1$, its leading term is $2^{11}\ell^{11}$. With $17!!=34459425$, the leading coefficient in the basis $(2\ell-5)!!$ is therefore $g_{-19}\,2^{11}/17!!=\frac{8}{15}=\frac{1008}{1890}$, the leading coefficient of $\mathcal A_5/1890$. The term $g_{-18}u^{-18}$ contributes $g_{-18}\,2^\ell\ell!\binom{\ell+8}{8}$, whose leading coefficient $-13440/8!=-\frac13=-\frac{1536}{4608}$ is that of $-\mathcal B_5/4608$.

Theorem~\ref{thm:k5} was checked independently in four ways. First, the sector $w=0$ of the symmetric recursion coincides, as a Laurent polynomial, with the generating function of $W_5$ obtained from Theorem~\ref{thm:recursion} alone. Second, exhaustive enumeration of $\cP_{1,5}$ and $\cP_{2,5}$ by canonical forms gives $1896$ and $86453$ networks, with defects $\frac{319}{8}$ and $1005$. It gives the same values for the involution sum of \eqref{eq:Dinv}, as Lemma~\ref{lem:exponent2} requires. Third, $W_5(\ell)+D_5(\ell)$ is an integer for every $\ell\le12$, although neither term is.

Fourth, the coefficients above satisfy the identities \eqref{eq:gtests} with $k=5$ and $s=9$. The first identity reads $17!!/(2^5\,5!)=34459425/3840=2297295/256$. The second reads $-13440\cdot\frac{256}{2297295}\cdot\frac{17!!}{2^9\,8!}=-\frac52$. The third reads
\[
\frac{81}{2}-\frac{1}{24}-\frac{17}{2}\cdot\frac{225225}{2297295}=\frac{81}{2}-\frac{1}{24}-\frac56=\frac{317}{8},
\]
which is the coefficient $\frac{50k^2-62k+11}{24}$ of Corollary~\ref{cor:threeterm} at $k=5$. The fourth reads $\frac{3465}{64}\cdot17\cdot15\cdot13=20\cdot\frac{2297295}{256}$, which is the constant of Theorem~\ref{thm:symlaw} at $k=5$. Both asymptotic constants are thereby confirmed independently of their proofs.

The $2^\ell\ell!$ sectors in Theorem~\ref{thm:k5} satisfy $\mathcal B_5/4608=B_{W_5}/4608+B_{D_5}/96$ coefficient by coefficient, where $B_{W_5}$ and $B_{D_5}$ are the polynomials multiplying $2^\ell\ell!$ in $W_5$ and $D_5$. The weighted count $W_5$ has no small-$\ell$ corrections, and the corrections at $\ell=1,2$ originate entirely in the defect.

Figure~\ref{fig:ratios} compares the exact counts with the leading-order laws of Section~\ref{sec:asympt} for $k\le5$ and $10\le\ell\le400$. The counts are computed from Lemma~\ref{lem:exactTC}, Theorem~\ref{thm:defects}, Corollary~\ref{cor:D4}, Theorem~\ref{thm:k5}, and the closed forms of $|\cP_{\ell,k}|$ for $k\le4$ given by the same evaluation. Panel (a) shows the non-tree-child excess normalised by Theorem~\ref{thm:gamma}, and panel (b) the defect density normalised by Theorem~\ref{thm:symlaw}. Both ratios tend to one. At $\ell=400$ they still differ from one by $1.7$ to $6.1$ per cent in panel (a) and by $4.7$ to $8.6$ per cent in panel (b). Deviations of this size are consistent with relative corrections of order $\ell^{-1/2}$ with coefficients increasing with $k$.

\begin{figure}[t]
\centering
\includegraphics[width=\textwidth]{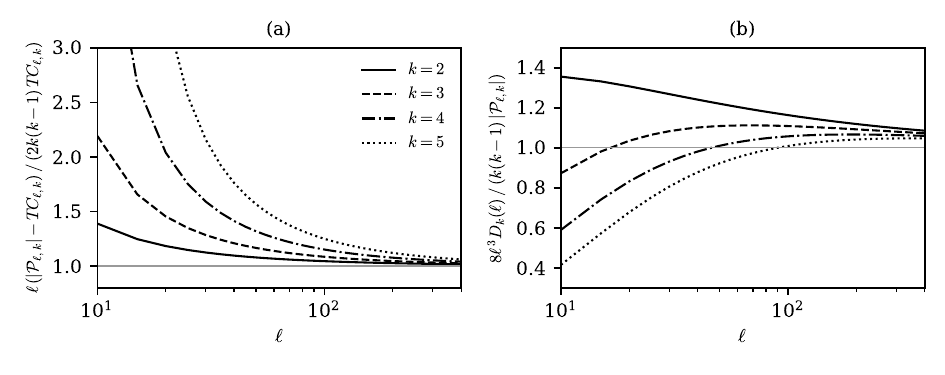}
\caption{(a) The ratio $\ell\,(|\cP_{\ell,k}|-TC_{\ell,k})/(2k(k-1)\,TC_{\ell,k})$ and (b) the ratio $8\ell^3D_k(\ell)/(k(k-1)\,|\cP_{\ell,k}|)$ for $k=2,3,4,5$, computed from the exact closed forms. By Theorems~\ref{thm:gamma} and~\ref{thm:symlaw} both tend to one as $\ell\to\infty$.}
\label{fig:ratios}
\end{figure}

\section{Six and more reticulations}\label{sec:beyond}

The restriction $k\le5$ in Sections~\ref{sec:symrec} and~\ref{sec:k45} enters only through Corollary~\ref{cor:Dinv}, which rests on Lemma~\ref{lem:exponent2}. For every $k$ one has $\sum_{\sigma\in\Aut(N)}|\Aut(N)|^{-1}=1$, and therefore
\begin{equation}\label{eq:burnside}
|\cP_{\ell,k}|=\sum_{N\in\cP_{\ell,k}}\ \sum_{\sigma\in\Aut(N)}\frac{1}{|\Aut(N)|}.
\end{equation}
The count of networks is thus the weighted count of pairs formed by a network and an arbitrary automorphism. For $k\le5$ the automorphisms are the identity and involutions, and \eqref{eq:burnside} reduces to \eqref{eq:QsigmaTop}. For $k\ge6$, Example~\ref{ex:dihedral} shows that elements of order four occur, and by Lemma~\ref{lem:autbound} every automorphism has order a power of two. An exact recursion valid for every $k$ therefore requires an extension of Theorem~\ref{thm:symrec} to automorphisms of order $2^r$ for every $r$. The present section describes what this extension involves; it is not carried out here.

The first change concerns the orbit data. An automorphism of order $2^r$ permutes the roots, the closed reticulations and the open reticulations in cycles whose lengths are powers of two. The pair $(a_f,a_s)$ is replaced by the cycle type of the action on the roots. The types $(c_f,c_x,c_s)$ and $(b_f,b_s)$ are replaced by the cycle types on the closed and open reticulations. For every cycle of closed reticulations of length $2^i$, the data must also record whether $\sigma^{2^i}$ fixes or exchanges the two slots of each of its reticulations. The prefactor of \eqref{eq:symrec} is replaced by the inverse order of the centraliser of a permutation with the given cycle type, namely $\prod_i n_i!\,2^{in_i}$ for $n_i$ cycles of length $2^i$, combined with the slot data.

The second change concerns the invariant layer forests. Lemma~\ref{lem:invforest} uses that a fixed vertex has either two fixed or two exchanged children. Under an automorphism of order $2^r$, a vertex has an orbit of length $2^i$, and its stabiliser, generated by $\sigma^{2^i}$, either fixes its children or exchanges them. The forks of Lemma~\ref{lem:invforest} are thereby replaced by a hierarchy of forks indexed by the level at which the children are exchanged. The block generating functions $G_R$ and $G_F$ are replaced by a finite system indexed by that level, obtained by iterating the construction of Lemma~\ref{lem:invforest}. Since every leaf is fixed, every moved subtree carries slots only, so these generating functions remain polynomials for fixed $k$.

The third change concerns placements and parallel pairs. Points are placed on edge orbits of every length $2^i$, and the number $w$ is replaced by the numbers of edge orbits of each length. Parallel pairs are classified by the length of their orbit and by whether the stabiliser exchanges their two edges. Lemma~\ref{lem:eqplace} then extends by inclusion--exclusion over these classes.

The smallest number of reticulations that carries an automorphism of order two is two, attained by $\Gamma_1$. For order four it is six, by Lemma~\ref{lem:exponent2} and Example~\ref{ex:dihedral}. The corresponding number for order eight is not determined here; below it, the extension concerns automorphisms of order at most four. The number of cycle types grows with $k$, but the recursion remains polynomial in $\ell$ for fixed $k$. Independent checks of it would be weaker than for $k\le5$, since exhaustive enumeration with six reticulations is practical only for one leaf by the method used here. Integrality of the counts and the identities \eqref{eq:gtests}, which hold for every $k$, would remain available as checks.

The extension would provide an exact finite recursion for every number of reticulations, and hence, by Corollary~\ref{cor:shape}, a closed form for each fixed $k$, each obtained by a separate finite computation. It would not provide a formula uniform in $k$. The asymptotic results of Section~\ref{sec:asympt} hold for every $k$ and do not depend on it.

\section{Conclusions}\label{sec:conclusions}

The number of unrestricted binary phylogenetic networks with a fixed number of reticulations has been separated into a weighted count, in which every network contributes the inverse order of its group of leaf-fixing automorphisms, and a defect supported on symmetric networks. The weighted count is the leaf-labelled counterpart of the vertex-labelled count. It satisfies a recursion over the source layers of the tree-component structure for every number of reticulations. The states of the recursion are the numbers of roots, reticulations and unbroken parallel pairs. Neither a list of component graphs nor a distinction between symmetric and asymmetric configurations enters. In the variable $u=\sqrt{1-2x}$ the recursion becomes an identity between Laurent polynomials, which determines the two-sector form of every count and its finitely many exceptional values.

Automorphisms of networks originate in two local seeds, a pair of tree nodes over two common reticulations and a pair of reticulations over a common reticulation. The automorphism groups are $2$-groups. They are elementary abelian for at most five reticulations, whereas a dihedral group of order eight occurs with six. For at most five reticulations the defect is therefore the weighted count of networks with a distinguished involution. The layer-peeling recursion extends to such pairs, because the parts of a network moved by an automorphism contain no leaves. The resulting closed forms, obtained by an exact symbolic evaluation and checked against exhaustive enumeration, integrality and the asymptotic constants, reproduce the published counts with up to four reticulations by an independent method. They give the count with five reticulations, together with its splitting into weighted count and defect, in which all exceptional small values originate in the defect. The comparison with the published three-reticulation generating function of Mansouri has identified a coefficient error in its numerator, which is corrected.

For every number of reticulations, uniformly over explicit ranges, the asymptotic analysis has determined three corrections to the first-order behaviour. Networks with a single violation of the tree-child condition dominate the non-tree-child networks. They comprise networks with a stacked pair of reticulations and networks with a tree node over two reticulations, in equal numbers to leading order. Together they amount to a fraction $2k(k-1)/\ell$ of the tree-child networks, which yields the third term of the asymptotic expansion. Only the second kind is reticulation-visible, so reticulation-visible networks exceed tree-child networks by the fraction $k(k-1)/\ell$; this law coincides with $k!/\ell$ only for $k=2,3$. A balance relation between seams, moved reticulations, splitters and moved vertices shows that the crossed cherry is the only symmetric configuration of leading order, so a random network is symmetric with probability $k(k-1)/(4\ell^3)$. All three laws are governed by the number of pairs of reticulations. The exact counts with up to five reticulations confirm the predicted constants independently of the proofs.

An exact recursion for every number of reticulations requires the extension of the symmetric recursion to automorphisms of order a power of two. Its ingredients have been identified: cycle types in place of involution types, a hierarchy of forks in the invariant layer forests, and placements on edge orbits of every length. The asymptotic results do not depend on this extension.

\appendix
\section{Symmetric networks with four reticulations}\label{app:k4}

This appendix classifies the symmetric networks with four reticulations by hand. The classification is independent of Section~\ref{sec:symrec}. Besides the defect $D_4$, it yields the numbers of networks whose automorphism group has order two or four. The seed argument of Proposition~\ref{prop:class}, carried one reticulation further, yields six configurations in addition to $\Gamma_1$, $\Gamma_2$ and $\Gamma_3$. Each is listed with the shape it contracts to.

The stacked cross $\Gamma_4$ consists of tree nodes $a,b,v,w$ and reticulations $y,y',r_1,r_2$, with edges $a\to y$, $a\to y'$, $b\to y$, $b\to y'$, $y\to v$, $y'\to w$ and $u\to r_i$ for $u\in\{v,w\}$. It contracts to a reticulation followed by a tree node.

The doubly propagated cross $\Gamma_5$ is $\Gamma_3$ with one further level of propagation and one further seam. It contracts to a chain of three tree nodes.

The branched cross $\Gamma_6$ consists of a tree node $p$ with children $q,q'$, where $q\to v,y$ and $q'\to w,y'$, the pair $v,w$ is a twin pair over seams $r_1,r_2$, and the pair $y,y'$ is a twin pair over seams $s_1,s_2$. It contracts to a tree node with two tree-node children.

The remaining three configurations arise from a seed of type (R), whose reticulations $v,w$ have the common reticulation child $c$, and they differ in how the parents of $v$ are placed.

\begin{itemize}
\item In the mixed twin reticulations $\Gamma_7$ one parent of $v$ is fixed. The configuration consists of tree nodes $a,p,q,q'$ and reticulations $v,w,c,z$, with edges $a\to v$, $a\to w$, $p\to q$, $p\to q'$, $q\to v$, $q\to z$, $q'\to w$, $q'\to z$, $v\to c$ and $w\to c$. It contracts to a tree node $u$ together with a reticulation $\rho$ of which $u$ is a parent. When the two in-edges of $\rho$ both leave $u$, the exit out of $z$ feeds the entry of $a$.
\item In the crossed twin reticulations $\Gamma_8$ both parents of $v$ are moved, and the second child of one of them is the image of the other. The configuration consists of tree nodes $p,s,s',a,a'$ and reticulations $z,v,w,c$, with edges $p\to s$, $p\to s'$, $s\to a$, $s\to w$, $s'\to a'$, $s'\to v$, $a\to z$, $a\to v$, $a'\to z$, $a'\to w$, $v\to c$ and $w\to c$.
\item In the triangular twin reticulations $\Gamma_9$ both parents of $v$ are moved and one is the parent of the other. The configuration consists of tree nodes $p,q,q',b,b'$ and reticulations $z,v,w,c$, with edges $p\to q$, $p\to q'$, $q\to b$, $q\to v$, $b\to v$, $b\to z$, $q'\to b'$, $q'\to w$, $b'\to w$, $b'\to z$, $v\to c$ and $w\to c$.
\end{itemize}

Both $\Gamma_8$ and $\Gamma_9$ have one entry and the two distinguishable exits $z$ and $c$, so each contracts to a tree node of a tree in two orientations. They arise in case (R) when both parents of $v$ are moved but a moved parent does not force a further seam, namely when $b\to\sigma(a)$ or $a\to b$ for the two parents $a,b$ of $v$. These two possibilities are easily overlooked in the seed argument, and the exhaustive enumeration confirms that they occur.

We write $\Delta$ for the configuration $\Gamma_2$ whose two entries leave a common tree node, which is $\Gamma_1$ with merged exits. We use the structures $\cP^*_{\ell,1}$ of Definition~\ref{def:Pstar}, in which a parallel pair may occur at the contraction site. Let $P_1=|\cP_{\ell,1}|$ and $P_2=|\cP_{\ell,2}|$. Let $A=(\ell-1)t_\ell$ be the number of symmetric networks in $\cP_{\ell,2}$, and $O=|\mathcal{OCP}_{\ell,1}|=(2\ell-2)!/(2^{\ell-1}(\ell-2)!)$. Let $R=P_1-\ell O$ be the number of networks of $\cP_{\ell,1}$ whose reticulation has a tree-node child. Counting by contraction as in Theorem~\ref{thm:defects} gives the thirteen classes of Table~\ref{tab:k4classes}.

Consistency checks are built into the list. The networks obtained by expanding the twin node $p'$ of a symmetric network of $\cP_{\ell,2}$ are exactly class E3. Those obtained by expanding its seams are exactly class E4. Those obtained by expanding a tree node outside its $\Gamma_1$ are exactly class E2, each counted twice.

\begin{table}[t]
\centering
\small
\begin{tabular}{lp{3.6cm}p{4.0cm}p{4.6cm}c}
\toprule
Class & Configuration & Contraction & Number & $|\Aut|$\\
\midrule
E1 & $\Delta$ and one reticulation & edge of $\cP^*_{\ell,1}$ & $(2\ell+2)P_1+(2\ell-1)t_\ell$ & 4\\
E2 & two disjoint $\Gamma_1$ & two tree nodes of a tree & $\binom{\ell-1}{2}t_\ell$ & 4\\
E3 & $\Gamma_4$, common entry node & tree node of a tree & $(\ell-1)t_\ell$ & 4\\
E4 & $\Gamma_3$, $r_1,r_2$ merged & tree node of a tree, two orientations & $2(\ell-1)t_\ell$ & 4\\
E5 & $\Gamma_1$ and two reticulations & tree node of $\cP_{\ell,2}$ & $(\ell+1)(P_2-A)+A$ & 2\\
E6 & $\Gamma_2$ and one reticulation & reticulation of $\cP_{\ell,2}$ & $2(P_2-A)$ & 2\\
E7 & $\Gamma_3$ and one reticulation & adjacent tree nodes of $\cP_{\ell,1}$ & $(\ell-1)P_1-R$ & 2\\
E8 & $\Gamma_4$, distinct entry nodes & reticulation and tree child in $\cP_{\ell,1}$ & $R$ & 2\\
E9 & $\Gamma_5$ & chain of three tree nodes of a tree & $(\ell-4)t_\ell+\ell t_{\ell-1}$ & 2\\
E10 & $\Gamma_6$ & tree node with two tree-node children & $\binom{\ell}{2}t_{\ell-1}-t_\ell$ & 2\\
E11 & $\Gamma_7$ & pair $(u,\rho)$ in $\cP^*_{\ell,1}$ & $2P_1+(2\ell-1)t_\ell$ & 2\\
E12 & $\Gamma_8$ & tree node of a tree, two orientations & $2(\ell-1)t_\ell$ & 2\\
E13 & $\Gamma_9$ & tree node of a tree, two orientations & $2(\ell-1)t_\ell$ & 2\\
\bottomrule
\end{tabular}
\caption{Symmetric classes with four reticulations. The counts of E9 and E10 hold for $\ell\ge2$; both vanish for $\ell=1$.}
\label{tab:k4classes}
\end{table}

\begin{theorem}\label{thm:D4}
For $\ell\ge2$ the numbers of networks in $\cP_{\ell,4}$ whose automorphism group has order four and order two are
\begin{align}
N_4(\ell)&=(2\ell+2)P_1+\tfrac12(\ell^2+7\ell-6)\,t_\ell,\label{eq:N4}\\
N_2(\ell)&=(\ell+3)P_2+(\ell+1)P_1-(\ell-2)(\ell-4)\,t_\ell+\tfrac12\ell(\ell+1)\,t_{\ell-1},\label{eq:N2}
\end{align}
no network has a larger group, and
\begin{equation}\label{eq:D4}
D_4(\ell)=\tfrac34N_4(\ell)+\tfrac12N_2(\ell).
\end{equation}
For $\ell=1$ one has $N_4=1$, $N_2=5$ and $D_4(1)=\tfrac{13}{4}$.
\end{theorem}

\begin{proof}
Equation \eqref{eq:N4} is the sum of the counts of classes E1--E4. Equation \eqref{eq:N2} is the sum of E5--E13 after simplification. The relation $E_7+E_8=(\ell-1)P_1$ and the identity $E_9+E_{10}=(\ell-5)t_\ell+\tfrac12\ell(\ell+1)t_{\ell-1}$, valid for $\ell\ge2$, are used. Completeness of the classes follows from the seed argument, extended to both placements of moved parents in case (R).
\end{proof}

The defect \eqref{eq:D4} coincides with Corollary~\ref{cor:D4} for $1\le\ell\le12$. The exhaustive enumeration for $\ell\le3$ reproduces the group-order counts \eqref{eq:N4}--\eqref{eq:N2}: $(N_4,N_2)=(1,5)$, $(18,99)$ and $(204,1767)$ for $\ell=1,2,3$.

\section*{Acknowledgements}
J.B. appreciates fruitful discussions with Maria Vallespir-Socias, Joana Rossell\'o, Margalida Batle, Maria del Mar Batle, and Regina Batle. J.B. received no funding for the present work.

\section*{Data availability}
The data and programs supporting this study are available from the author on reasonable request.

\end{document}